\documentclass[11pt]{article}
\usepackage{epsfig}
\usepackage{amssymb,amsmath,amsthm,amscd}
\usepackage{latexsym}

\newtheorem{thm}{Theorem}[section]
\newtheorem{prop}[thm]{Proposition}
\newtheorem{cor}[thm]{Corollary}
\newtheorem{lem}[thm]{Lemma}

\theoremstyle{definition}
\newtheorem{defn}[thm]{Definition}

\newcounter{labelflag} 

\newcommand{\Label}[1]{
                       \ifnum\thelabelflag=1
                          \ifmmode
                             \makebox[0in][l]{\qquad\fbox{\rm#1}}
                          \else
                             \marginpar{\vspace{0.7\baselineskip}
                                        \hspace{-1.1\textwidth}
                                        \fbox{\rm#1}}
                          \fi
                       \fi
                       \label{#1}
                      }

\def \thtwot {{  \theta_{t}  }}

\def\R{\mathbb{R}}

\newcommand{\be}{\begin{equation}}
\newcommand{\ee}{\end{equation}}

\def \calf {{  {\mathcal{F}} }}

\def \calftwo {{  {\mathcal{F}}  }}
\def \cala {{  {\mathcal{A}}  }}
\def \calb {{  {\mathcal{B}}  }}
\def \cald {{  {\mathcal{D}}  }}
\def \caln {{  {\mathcal{N}}  }}

\begin{document}

\baselineskip=1.38\baselineskip

\begin{titlepage}
\title{\large  \bf \baselineskip=1.3\baselineskip   
Stochastic Bifurcation of Pathwise  Random Almost 
Periodic and  Almost Automorphic Solutions
for Random Dynamical Systems}

\vspace{10 mm}

\author{ 
         Bixiang Wang  \vspace{4mm}\\
Department of Mathematics, \ 
 New Mexico Institute of Mining and Technology  \vspace{1mm}\\ 
Socorro,  NM~87801, USA  \vspace{3mm}\\
Email:    bwang@nmt.edu  }

\date{}
\end{titlepage}

\maketitle

\noindent
 
\begin{abstract}  \baselineskip=1.3\baselineskip
In this paper, we introduce concepts of  
pathwise random almost periodic and 
   almost automorphic solutions 
 for  dynamical systems generated
 by non-autonomous stochastic equations.
 These solutions are pathwise  stochastic analogues of deterministic 
 dynamical systems. The existence and bifurcation of random
 periodic (random almost periodic, 
 random almost automorphic) solutions have been established
 for a one-dimensional   stochastic equation
 with multiplicative noise.
\end{abstract}



\section{Introduction}
\setcounter{equation}{0}

 This paper is concerned with almost periodic
 and almost automorphic  dynamics of random
 dynamical systems associated with
 stochastic differential equations 
 driven by time-dependent deterministic forcing.
 We will first define pathwise random almost periodic
 solutions  and    almost automorphic solutions 
 for such systems, which are 
 special  cases of random complete solutions
 and random complete quasi-solutions.
 We then study existence,     stochastic 
 pitchfork and transcritical bifurcation of  these
 types of   solutions
 for 
 one-dimensional non-autonomous stochastic equations.

Almost  periodic  and almost automorphic solutions of
 deterministic differential equations
  have been extensively studied by
many experts, see, e.g.,  
 \cite{fin1,  lev1, liu1,  sel1, she1, she2, she3, she4, she5,  wangyi1,  war1,
 yos1, zai1}
and the references therein.
In particular,  the $\omega$-limit sets of such solutions
have been investigated  in \cite{ liu1, she1, 
she2, she3, she4, she5, wangyi1}.
However, as far as   the author is  aware, it seems that
there is no result  available
in the literature on existence and stability of
{\it pathwise}  random almost periodic or almost automorphic
solutions for stochastic equations.
The first goal of the present paper is to introduce these concepts
for random dynamical systems generated by
non-autonomous stochastic equations.
Roughly speaking, a pathwise random almost periodic
(almost automorphic)  solution is a  random complete
quasi-solution which is pathwise almost periodic (almost automorphic)
(see Definitions \ref{comporbit} and \ref{persol}
below).
It is worth mentioning  that 
    a pathwise random almost periodic
(almost automorphic)  solution is  actually  {\it  not }
a solution of the system for a fixed sample path,
and it is just a    complete
quasi-solution 
in the sense of Definition \ref{comporbit}.
In this paper, in addition to existence of pathwise
random periodic (almost periodic, almost automorphic)
solutions, we will also study  the stability and bifurcation
of these solutions.  More precisely, we 
will investigate stochastic pitchfork bifurcation
of the  one-dimensional non-autonomous  equation
 \be
 \label{intr1}
 {\frac {dx}{dt}} =\lambda x -\beta (t) x^3
 + \gamma (t,x) + \delta x  \circ {\frac {d\omega}{dt}},
 \ee
 and transcritical bifurcation of the equation
  \be
 \label{intr2}
 {\frac {dx}{dt}} =\lambda x -\beta (t) x^2
 + \gamma (t,x) + \delta x  \circ {\frac {d\omega}{dt}},
 \ee
 where  
   $\lambda$ and $\delta$  
 are constants,    $\beta: \R \to \R$  
 is positive, and $\gamma: \R \times \R \to \R$ satisfies
 some growth conditions.
 The stochastic equations \eqref{intr1} and \eqref{intr2}
 are understood in the sense of Stratonovich integration.
 
 In the deterministic case (i.e., $\delta =0$), 
 these equations are classical examples for  
 demonstrating pitchfork and transcritical bifurcation
 of fixed points.
 In the stochastic case with constant $\beta =1$  and
 $\gamma =0$, the stochastic bifurcation
 of stationary solutions 
 and invariant measures  of \eqref{intr1}-\eqref{intr2}
 has been studied  in \cite{arn1, arn2}.
 In the real noise case, the same problem was discussed
 in \cite{xu1}.
 When $\beta=1$ and $\gamma$ 
 is time-independent,
 the bifurcation of stationary solutions of
 \eqref{intr1}-\eqref{intr2} 
   was examined in \cite{arn1, arn3}.
   For the bifurcation of stationary solutions
   of
   \eqref{intr1}
   with additive noise, we refer the reader to
   \cite{cra2}.
   It seems   that the bifurcation problem of
   \eqref{intr1} and \eqref{intr2}
   has not been studied in the literature
   when $\beta$ and $\gamma$ are time-dependent.
   The purpose of this paper
   is to investigate this problem and explore bifurcation
   of pathwise random complete solutions including
   random periodic (almost periodic, almost  automorphic)
   solutions.  Actually,  for time-dependent $\beta$
   and $\gamma$ satisfying certain conditions,
   we prove the pathwise complete quasi-solutions of
   \eqref{intr1} undergo a stochastic pitchfork bifurcation
   at $\lambda =0$: 
     for $\lambda \le 0$,  the zero solution      is the 
     unique   random complete quasi-solution
     of \eqref{intr1} which is pullback asymptotically stable
     in $\R$;  for   $\lambda>0$,    the zero solution 
     is unstable and two more  tempered
     random complete quasi-solutions
     $x^+_\lambda >  0$  and $x^-_\lambda  < 0$
     bifurcate from zero, i.e., 
   $$
     \lim_{\lambda \to 0}
     x^\pm_\lambda (\tau, \omega) = 0,
     \quad \text{  for all } \ \tau \in \R
     \ \text{  and }  \omega \in \Omega.
    $$
    The tempered random attractor $\cala_\lambda$
     of \eqref{intr1} is trivial for $\lambda \le 0$,
     and  is  given by 
     $\cala_\lambda =\{  
      [x^-_\lambda (\tau, \omega), 
     x^+_\lambda (\tau, \omega) ]
      : \tau \in \R,
     \omega \in \Omega \}$.
     If, in addition, $\beta$  and $\gamma$
     are both $T$-periodic in time  for some $T>0$,
     then $x^-_\lambda$
     and  $x^+_\lambda$ are also $T$-periodic.
     In this case,  we obtain pitchfork bifurcation
     of pathwise random periodic solutions of \eqref{intr1}.
     It seems   that the bifurcation of almost periodic and
     almost automorphic solutions is much more involved.
     Nonetheless, for $\gamma =0$, we will prove
     if $\beta$ is almost periodic (almost automorphic),
     then so are $x^-_\lambda$
     and  $x^+_\lambda$.  As a consequence,
     we obtain  stochastic pitchfork bifurcation
     of pathwise random  almost periodic
     (almost automorphic) solutions
     of \eqref{intr1}   in this case.
     By similar arguments,  we will establish stochastic
     transcritical bifurcation of pathwise 
     random complete quasi-solutions of \eqref{intr2}
     at $\lambda =0$.
     If $\gamma =0$  and $\beta$ is periodic (almost periodic,
     almost automorphic),  we further 
     establish  the transcritical
     bifurcation of random periodic (almost periodic, almost
     automorphic) solutions of  \eqref{intr2}  (see 
     Corollary \ref{2thmst} and Theorem \ref{3thmst}).

   This  paper is organized as follows.
   In the next section,  we introduce the concepts
   of pathwise random almost periodic (random almost automorphic) solutions
   for random dynamical systems
   (for  pathwise random periodic solutions,  the definition can be found
   in   \cite{fen2, zhao1}).
       We will also
   review some results regarding pullback attractors.
   In the last two sections, we prove
   stochastic pitchfork bifurcation
   and transcritical bifurcation
   for 
   equations
   \eqref{intr1}
   and \eqref{intr2},
   respectively.

\section{Preliminaries}
\setcounter{equation}{0}

 In this   section,  
   we  introduce   concepts  of
  pathwise    random 
almost periodic  and  
 almost automorphic solutions 
for random dynamical systems  generated by
 differential equations
driven  simultaneously
by  non-autonomous deterministic
and stochastic forcing.
We also  review  some known
results 
%from \cite{wan5}
regarding random attractors
for non-autonomous stochastic 
equations.
The reader is further  referred   to  
  \cite{bat1,  bat2,  car3, car4, car5, 
 car6, car7,  chu1, 
   cra1, dua3,   fla1,  gar1,  huang1,
  schm1} for  autonomous    random 
attractors  
 and to \cite{bab1, hal1, sel2, tem1}
for deterministic attractors.

Let $(X, d)$  be   a complete
separable  metric space 
and 
  $(\Omega, \calftwo, P,  \{\thtwot\}_{t \in \R})$
be a   metric dynamical system as in \cite{arn1} .
Given a subset  $A$ of $X$, 
the neighborhood of   $A$ with
radius $r >0$ is denoted by $\caln_r(A)$.
  A mapping $\Phi$: $ \R^+ \times \R \times \Omega \times X
\to X$ is called a continuous  cocycle on $X$
over $ \R$
and
$(\Omega, \calftwo, P,  \{\thtwot\}_{t \in \R})$
if   for all
  $\tau\in \R$,
  $\omega \in   \Omega $
  and    $t, s \in \R^+$,  the following conditions (i)-(iv)  are satisfied:
\begin{itemize}
\item [(i)]   $\Phi (\cdot, \tau, \cdot, \cdot): \R ^+ \times \Omega \times X
\to X$ is
 $(\calb (\R^+)   \times \calftwo \times \calb (X), \
\calb(X))$-measurable;

\item[(ii)]    $\Phi(0, \tau, \omega, \cdot) $ is the identity on $X$;

\item[(iii)]    $\Phi(t+ s, \tau, \omega, \cdot) =
 \Phi(t,  \tau +s,  \theta_{s} \omega, \cdot) 
 \circ \Phi(s, \tau, \omega, \cdot)$;

\item[(iv)]    $\Phi(t, \tau, \omega,  \cdot): X \to  X$ is continuous.
    \end{itemize}
   Such  $\Phi$ is called  
a  continuous periodic  cocycle
with period $T$ if  
$\Phi(t,  \tau +T, \omega, \cdot)
= \Phi(t, \tau,  \omega, \cdot )$
  for every $t  \in \R^+$,
     $\tau \in \R$  and $\omega \in \Omega$.
Let   $\cald$   be 
 a  collection  of  some families of  nonempty subsets of $X$:
\be
\label{defcald}
{\cald} = \{ D =\{ D(\tau, \omega ) \subseteq X: \ 
D(\tau, \omega ) \neq \emptyset,  \ 
  \tau \in \R, \
  \omega \in \Omega\} \}.
\ee
We now define $\cald$-complete  solutions  for $\Phi$.

\begin{defn}
\label{comporbit}
 Let $\cald$ be a collection of  families of
 nonempty  subsets of $X$  given by \eqref{defcald}.

      (i)   A mapping $\psi: \R \times \R \times \Omega$
 $\to X$ is called a complete orbit (solution) of $\Phi$ if for every  $t \in \R^+ $,
 $s, \tau \in \R$ and $\omega \in \Omega$,  the following holds:
$$
 \Phi (t,  \tau +s, \theta_{s} \omega,
  \psi (s, \tau, \omega) )
  = \psi (t +  s, \tau, \omega ).
$$
 If, in  addition,    there exists $D=\{D(\tau, \omega): \tau \in \R,
 \omega \in \Omega \}\in \cald$ such that
 $\psi(t, \tau, \omega)$ belongs to
 $D ( \tau +t, \theta_{ t} \omega )$
 for every  $t \in \R$, $\tau \in \R$
 and $\omega \in \Omega$, then $\psi$ is called a
 $\cald$-complete orbit (solution)  of $\Phi$.

      (ii)   A mapping $\xi:  \R   \times \Omega$
 $\to X$ is called a complete quasi-solution  of $\Phi$ if for every  $t \in \R^+ $,
 $\tau \in \R$ and $\omega \in \Omega$,  the following holds:
$$
 \Phi (t,  \tau,  \omega,
  \xi (\tau, \omega) )
  = \xi (  \tau +t,  \theta_t \omega ).
$$
 If, in  addition,    there exists $D=\{D(\tau, \omega): \tau \in \R,
 \omega \in \Omega \}\in \cald$ such that
 $\xi (\tau, \omega)$ belongs to
 $D ( \tau,  \omega )$
 for all   $\tau \in \R$
 and $\omega \in \Omega$, then $\xi$ is called a
 $\cald$-complete   quasi-solution  of $\Phi$.
 
 \end{defn}
 
 \begin{defn}
 \label{persol}
 Let  $\xi:  \R   \times \Omega$
 $\to X$   be a mapping.
 
  (i)   $\xi$  is called a  random periodic function 
 with period $T$  if  
 $\xi (\tau  +T, \omega) = \xi (\tau, \omega)$
 for  every  $\tau \in \R$  and $\omega \in \Omega$.

      (ii)  $\xi$  is called a  random almost periodic function   
    if   for  every $\omega \in \Omega$  and $\varepsilon>0$, there
 exists $l = l(\omega, \varepsilon)>0$ such that every
 interval of length $l$ contains  a number $t_0$
 such that
 $$
 |\xi(\tau +t_0, \omega) - \xi (\tau, \omega)|
 < \varepsilon, \quad  \text{ for all } \  \tau \in \R.
 $$ 
   
         (iii)   $\xi$
   is called a  random almost automorphic  function    
    if   for
 every $\omega \in \Omega$  and  every  sequence $\{\tau_n\}_{n=1}^\infty$, there
 exist  a subsequence    $\{\tau_{n_m}\}_{m=1}^\infty$
 of  $\{\tau_n\}_{n=1}^\infty$ and a  map $\zeta^\omega: 
 \R \to X$ such that for all $\tau \in \R$,
 $$
 \lim_{m \to \infty}
 \xi (\tau + \tau_{n_m}, \omega)
 = \zeta^\omega (\tau)
 \quad \text{ and } \
  \lim_{m \to \infty}
 \zeta^\omega  (\tau - \tau_{n_m})
 = \xi (\tau, \omega).
 $$
 
 If $\xi$  is  a complete quasi-solution of $\Phi$
 and is also a 
   random periodic  (random almost periodic,
 random almost automorphic) function,
 then $\xi$ is 
     called a 
  random  periodic    (random almost periodic,
 random almost automorphic)  solution of $\Phi$.
   \end{defn}

Notice  that pathwise
random periodic solution  was introduced 
 in \cite{zhao1}.
We here further extend the concepts of
deterministic almost periodic and almost
automorphic solutions to the stochastic case.

 \begin{defn}
 \label{fixedpt}
 Let  $x_0 \in X$   and $E$ be a subset of $X$.
 Then $x_0$ is called a fixed point of $\Phi$  if 
 $\Phi (t, \tau, \omega, x_0) = x_0$ for all
 $t \in \R^+$,  $\tau \in \R$  and $\omega \in \Omega$.
 A fixed point $x_0$ is said to be  pullback Lyapunov stable in $E$
 if  for every $\tau \in \R$, $\omega \in \Omega$
 and $\varepsilon >0$, there exists  $\delta
 =\delta(\tau, \omega, \varepsilon)>0$ such that
 for all $t \in \R^+$,
 $$
 \Phi (t, \tau -t, \theta_{-t} \omega,
 \caln_\delta (x_0) \cap E )
 \subseteq \caln_\varepsilon (x_0) \cap E.
 $$
 If $x_0$  is not    pullback Lyapunov
 stable in $E$,   then $x_0$ is said
 to be pullback Lyapunov unstable in $E$.
 A fixed point $x_0$ is said to be  pullback 
 asymptotically  stable in $E$
 if it is pullback Lyapunov  stable in $E$
 and  for all $\tau \in \R$,
 $\omega \in \Omega$
 and $x \in E$,
 $$
 \lim_{t \to \infty}
\Phi (t, \tau -t, \theta_{-t} \omega,
 x) = x_0.
 $$
 \end{defn}

\begin{defn} 
\label{temset} 
Let
$D=\{D(\tau, \omega): \tau \in \R, \omega \in \Omega \}$
be a family of    nonempty subsets of $X$. 
We say $D$ is tempered in $X$ 
with respect to
$(\Omega, \calftwo, P,  \{\thtwot\}_{t \in \R})$
if  there exists $x_0 \in X$ such that for every  $c>0$,
$\tau \in \R$ and $\omega \in \Omega$,
$$
\lim_{t \to -\infty}
e^{c t} d (D( \tau +t, \thtwot \omega), x_0 )
=0.
$$
\end{defn}

\begin{defn}
\label{defatt}
 Let  
 $\cala = \{\cala (\tau, \omega): \tau \in \R,
  \omega \in \Omega \} \in \cald $  with $\cald$
  given by \eqref{defcald}. 
Then     $\cala$
is called a    $\cald$-pullback    attractor  of 
  $\Phi$
if  for all $t\in \R^+$, 
 $\tau \in \R$ and
 $\omega \in \Omega$, 
\begin{itemize}
\item [(i)]   $\cala$ is measurable  and
 $\cala(\tau, \omega)$ is compact.

\item[(ii)]   $\cala$  is invariant:
$ \Phi(t, \tau, \omega, \cala(\tau, \omega)   )
= \cala ( \tau +t, \theta_{t} \omega
) $.

\item[(iii)]   $\cala$
attracts  every  member    $B \in  \cald$, 
$$ \lim_{t \to  \infty} d (\Phi(t, \tau -t,
 \theta_{-t}\omega, B(\tau -t, 
 \theta_{-t}\omega) ) , \cala (\tau, \omega ))=0.
$$
 \end{itemize}
 If,  further,  there exists $T>0$ such that
 $$
 \cala(\tau +T, \omega) = \cala(\tau,    \omega ),
 \quad \forall \  \tau \in \R, \forall \
  \omega \in \Omega,
 $$
 then   $\cala$ is  called  a periodic attractor  with period $T$.
\end{defn}
 
 We recall the 
    following result  from   \cite{wan5, wan7}. 
Similar results  can be found in 
\cite{bat1,  cra1, fla1, schm1}.

\begin{prop}
\label{att}  
 Let $\cald$ be an    inclusion  closed  
 collection of some  families of   nonempty subsets of
$X$,  and $\Phi$  be a continuous   cocycle on $X$
over $\R$
and
$(\Omega, \calftwo, P,  \{\thtwot\}_{t \in \R})$.
Then
$\Phi$ has a  $\cald$-pullback
attractor $\cala$  in $\cald$
if 
$\Phi$ is $\cald$-pullback asymptotically
compact in $X$ and $\Phi$ has a  closed
   measurable 
     $\cald$-pullback absorbing set
  $K$ in $\cald$.
  The $\cald$-pullback
attractor $\cala$   is unique   and is  characterized   by,
for each $\tau  \in \R$   and
$\omega \in \Omega$,
$$
\cala (\tau, \omega)
 =\{\psi(0, \tau, \omega): \psi  \text{ is a   }  \cald  \text{-}
 \text{complete  solution  of } \Phi \} 
 $$
 $$
 =\{\xi( \tau, \omega): \xi  \text{ is a   }  \cald  \text{-}
 \text{complete   quasi-solution  of } \Phi \} .
$$
  \end{prop}

  \section{Pitchfork bifurcation of  stochastic equations  }
\setcounter{equation}{0}
 
 In this section,  we discuss pitchfork bifurcation
 of  the following one-dimensional
   stochastic equation with deterministic
 non-autonomous forcing:
 \be
 \label{geq1}
 {\frac {dx}{dt}} =\lambda x -\beta (t) x^3
 + \gamma (t,x) + \delta x  \circ {\frac {d\omega}{dt}},
 \quad  x(\tau) = x_\tau ,
 \ee
 where $\tau \in \R$, $t > \tau$, $x \in \R$, 
   $\lambda$ and $\delta$  
 are constants with $\delta>0$. 
 The function  $\beta: \R \to \R$    in \eqref{geq1}
 is    smooth  and      positive.  In addition, we   assume 
  there  exist $\beta_1 \ge \beta_0>0$ such   that   
 \be
 \label{bcon1}
 \beta_0 \le \beta (t) \le \beta_1 \quad
 \text{for   all } \  t \in \R.
 \ee
 The function $\gamma: \R \times \R 
 \to \R$ 
 in \eqref{geq1}
 is smooth and there exist two nonnegative numbers
 $c_1$  and $c_2$   with $c_1 \le c_2 <\beta_0$
 such that
\be\label{gacon1}
 c_1x^4 \le   \gamma (t,x)    x \le c_2 x^4
 \quad \text{  for all } \  t \in \R
 \quad  \text{ and } \  x \in \R.
 \ee
 Note   that condition
 \eqref{gacon1} implies that 
 $\gamma (t,0) = 0$
 for   all   $t \in  \R$.
 Therefore, 
 $x=0$ is a  fixed point of equation
 \eqref{geq1}.  
 The stochastic 
 equation \eqref{geq1} is understood in the
 sense of Stratonovich integration
 with  $\omega$  being   a two-sided real-valued
 Wiener process  on   the 
  probability space
   $(\Omega,\mathcal{F},P)$,  where 
$$ \Omega=\left\{\omega \in
 C(\mathbb{R},\mathbb{R} ):\omega(0)=0\right\} ,
$$ 
   $\mathcal{F}$ 
   is  the Borel $\sigma$-algebra induced by the 
   compact open topology of $\Omega$, 
    and $P$ is  the  Wiener measure on
     $(\Omega, \cal{F})$.
     There is a   classical 
  group $\{\theta_{t}\}_{t\in \R}$   acting
 on $(\Omega,\mathcal{F},P)$
 which  is given  by 
$$
 \theta_{t}\omega(\cdot)
 =\omega( \cdot +t)  - \omega(t), \
 \text{  for   all }   \omega\in\Omega
 \text{  and } t\in \mathbb{R}.
 $$
  It follows   from
  \cite{arn1}
  that  $(\Omega, \calf, P,   \{\theta_{t}\}_{t\in \R}  )$ is a  parametric 
dynamical  system and       
  there exists a $ \theta_{t}$-invariant set 
  $\tilde{\Omega}\subseteq \Omega$
of full $P$ measure  such that
for each $\omega \in \tilde{\Omega}$, 
\be\label{asyom}
{\frac  {\omega (t)}{t}} \to 0 \quad \mbox {as } \ t \to \pm \infty,
\ee
and
\be
\label{alzero}
\int_{-\infty}^0 e^{2 \delta \omega (s)} ds
=\infty
\quad  \text{ and } \quad 
\int^{\infty}_0 e^{2 \delta \omega (s)} ds
=\infty.
\ee
In the sequel,  we  only consider   
$\tilde{\Omega}$ rather  than 
$\Omega$,  and hence we will
 write     $\tilde{\Omega}$ as 
$\Omega$  for   convenience.

Under conditions \eqref{bcon1}
and \eqref{gacon1},  one can prove as
in \cite{arn1}  that
   the stochastic equation
\eqref{geq1}  has
a unique measurable solution
$x$ for  a   given
  initial value.
Moreover, 
for  every  $ \tau  \in \R$,  $\omega \in \Omega$
and $x_\tau \in \R$,  the solution
$x(\cdot, \tau, \omega, x_\tau)
\in C([\tau, \infty), \R)$  and is continuous in 
  $x_\tau$.   
Therefore,  one can  
      define a cocycle
       $\Phi: \R^+ \times \R \times  \Omega \times \R$
$\to  \R$ for     equation 
\eqref{geq1}. 
Given $t \in \R^+$,  $\tau \in \R$, $\omega \in \Omega$ 
 and $x_\tau \in \R$,
let 
 \be \label{phi}
 \Phi (t, \tau,  \omega, x_\tau) = 
  x (t+\tau,  \tau, \theta_{ -\tau} \omega, x_\tau) . 
\ee
By \eqref{geq1}, 
one can check   that   for every
$t \ge 0$, $\tau \ge 0$, $r \in \R$  and 
$\omega \in \Omega$,  
$$
\Phi (t + \tau, r, \omega,  \cdot )
=\Phi  (t , \tau + r,  \theta_{\tau}\omega,  \cdot)
\circ 
 \Phi (\tau, r, \omega,  \cdot) . 
$$
Since the solution of  \eqref{geq1}
is measurable in $\omega \in \Omega$
and continuous in initial data, we find that 
   $\Phi$ given by \eqref{phi} 
is a continuous   cocycle on $\R$ 
over  
$(\Omega, \calf,
P, \{\theta_t \}_{t\in \R})$.
We will study the dynamics of $\Phi$ in  this section.
   
    Given  a  bounded nonempty  subset 
 $I$  of $\R$,   we write 
   $  \| I \| = \sup \{ 
   | x  |:  x \in I \}$. 
Let $D =\{ D(\tau, \omega): \tau \in \R, \omega \in \Omega \}$  
   be   a   family of
  bounded nonempty   subsets of $\R $.  
  Recall that 
  $D$ is tempered if 
  for every  $c>0$, $\tau \in \R$   and $\omega \in \Omega$, 
 \be
 \label{temp1}
 \lim_{t \to  - \infty} e^{  c  t} 
 \| D( \tau  +t, \theta_t  \omega ) \|  =0. 
 \ee 
 Denote by  
  $\cald$        the  collection of all  tempered families of
bounded nonempty  subsets of $\R$, i.e.,
 \be
 \label{tempd}
\cald = \{ 
   D =\{ D(\tau, \omega): \tau \in \R, \omega \in \Omega \}: \ 
 D  \ \mbox{satisfies} \  \eqref{temp1} \} .
\ee

 In the  next subsection,  we consider the
 bifurcation problem of \eqref{geq1} 
 when $\gamma$ is  absent.  In this case,
 the    stochastic
 equation \eqref{geq1}  is exactly solvable
 which makes it  possible for one 
   to  completely determine  its
   dynamics.
  We will show the random complete
  quasi-solutions of \eqref{geq1}  undergo
 a  pitchfork
 bifurcation when $\lambda$ crosses zero from below.
 When $\beta$ is periodic (almost periodic, almost
 automorphic),  we show the random periodic 
 (random almost periodic,   random almost
 automorphic)  solutions  have similar bifurcation
 scenarios.
 We  finally   investigate pitchfork bifurcation
 of \eqref{geq1} with   $\gamma$ satisfying
 \eqref{gacon1}.

  \subsection{Pitchfork bifurcation of a typical 
  non-autonomous stochastic equation  }

This   subsection is devoted   to
   pitchfork
 bifurcation   of  \eqref{geq1} 
 without   $\gamma$.
 In other words,  we  consider 
 the  following non-autonomous stochastic    equation:
 \be
 \label{seq1}
 {\frac {dx}{dt}} =\lambda x -\beta (t) x^3
  + \delta x  \circ {\frac {d\omega}{dt}},
 \quad  x(\tau) = x_\tau ,
 \quad t>\tau .
 \ee
 As in the deterministic case,  equation
 \eqref{seq1} is exactly solvable. 
 To   find a solution  of \eqref{seq1},   one  may  introduce
 a new variable  $y = x^{-2} $
   for $x \neq 0$.
   Then $y$ satisfies 
   \be
   \label{seq2}
   {\frac {dy}{dt}}
   + 2 \lambda y
   = 2 \beta (t) - 2 \delta y \circ {\frac {d\omega}{dt}},
   \quad y(\tau) = y_\tau = x^{-2}_\tau, \quad t>\tau. 
   \ee
   For every  $t, \tau  \in \R$
   with $t \ge \tau$, $\omega \in \Omega$
   and $y_\tau \in \R$, by \eqref{seq2} we get
  $$
   y(t, \tau, \omega, y_\tau)
   = e^{2\lambda (\tau -t)
   + 2 \delta (\omega (\tau) - \omega (t) )} y_\tau
   + 2 \int_\tau^t e^{2 \lambda (r-t) + 2\delta (\omega (r)
   -\omega (t) ) } \beta (r) dr.
  $$
 Therefore,   the solution $x$ of \eqref{seq1}
  is given by,
  for  every  $t, \tau  \in \R$
   with $t \ge \tau$, $\omega \in \Omega$
   and $x_\tau \in \R$,
   \be
   \label{seq4}
   x(t, \tau, \omega, x_\tau)
   = {\frac {x_\tau} {
   \sqrt{
     e^{2\lambda (\tau -t)
   + 2 \delta (\omega (\tau) - \omega (t) )} 
   + 2 x^2_\tau \int_\tau^t e^{2 \lambda (r-t) + 2\delta (\omega (r)
   -\omega (t) ) } \beta (r) dr
   }}}.
   \ee
   It follows   from
   \eqref{seq4} that, 
   for each $t \in \R^+$, $\tau \in \R$,
  $ \omega \in \Omega$ 
  and $x_{0} \in \R$, 
     \be
   \label{seq5}
   x(\tau, \tau -t,  \theta_{-\tau}\omega,  x_{0} )
   = {\frac {x_{0}} {
   \sqrt{
     e^{ - 2\lambda  t 
   + 2 \delta   \omega (-t) } 
   + 2 x^2_{0} \int_{\tau -t}^\tau 
    e^{2 \lambda (r-\tau) + 2\delta  \omega (r -\tau)
     } \beta (r) dr
   }}}.
   \ee
   By \eqref{asyom}  and
   \eqref{seq5} we get,
   for every $\lambda >0$  and  $x_{0}>0$,
   $$
  \lim_{t \to \infty}
   x(\tau, \tau -t,  \theta_{-\tau}\omega,  x_{0} )
   = {\frac {1} {
   \sqrt{ 
     2   \int_{-\infty}^\tau 
    e^{2 \lambda (r-\tau) + 2\delta  \omega (r -\tau)
     } \beta (r) dr
   }}}.
   $$
   That is,   for    every
   $\lambda >0$  and  $x_{0}>0$,
   we have 
       \be
   \label{seq6}
      \lim_{t \to \infty}
   x(\tau, \tau -t,  \theta_{-\tau}\omega,  x_{0} )
   =
    {\frac {1} {
   \sqrt{ 
     2   \int_{-\infty}^0
    e^{2 \lambda r + 2\delta  \omega (r)
     } \beta (r +\tau) dr
     }}}.
   \ee
   Note   that
   the right-hand side of
   \eqref{seq6} is well defined 
   in terms of \eqref{bcon1}   and
   \eqref{asyom}. 
   Similarly, by \eqref{seq5}  we  obtain,
     for    every
   $\lambda >0$  and  $x_{0}<0$,
    \be
   \label{seq7}
      \lim_{t \to \infty}
   x(\tau, \tau -t,  \theta_{-\tau}\omega,  x_{0} )
   = 
    {\frac {-1} {
   \sqrt{ 
     2   \int_{-\infty}^0
    e^{2 \lambda r + 2\delta  \omega (r)
     } \beta (r +\tau) dr
     }}}.
     \ee
     Given  $\lambda >0$, $\tau \in \R$  and
     $\omega \in \Omega$, define
     \be
     \label{xplus}
     x^+_\lambda(\tau, \omega) 
      =
    {\frac {1} {
   \sqrt{ 
     2   \int_{-\infty}^0
    e^{2 \lambda r + 2\delta  \omega (r)
     } \beta (r +\tau) dr
     }}}
     \quad
     \text{ and } \quad
       x^-_\lambda (\tau, \omega) 
      =
    {\frac {-1} {
   \sqrt{ 
     2   \int_{-\infty}^0
    e^{2 \lambda r + 2\delta  \omega (r)
     } \beta (r +\tau) dr
     }}}.
     \ee
     It is evident   that for each
     $\lambda>0$ and 
     $\tau \in \R$, 
     both $x^+_\lambda(\tau, \cdot)$
     and $x^-_\lambda(\tau, \cdot)$ are measurable.
     We next prove   that 
     $x^+_\lambda$  and $x^-_\lambda$
     are  random complete quasi-solutions of
     equation \eqref{geq1}. 
     
     \begin{lem}
     \label{scso1}
     Suppose  \eqref{bcon1}  holds.
     Then for every $\lambda >0$,  $x^+_\lambda$
     and $x^-_\lambda$ given by
     \eqref{xplus} are  tempered random  complete
     quasi-solutions  of equation \eqref{geq1}.
     Moreover,   $x^+_\lambda$ is the only complete
     quasi-solution   in $(0, \infty)$
     with tempered reciprocal,
      and
      $x^-_\lambda$ is the only complete
     quasi-solution  in $(-\infty, 0)$
      with tempered reciprocal.
     \end{lem}
     
     \begin{proof}
     We first  prove 
     $x^+_\lambda$
     and $x^-_\lambda$   are random  complete
     quasi-solutions.
     Given $t \in \R^+$,
     $\tau \in \R$ and $\omega \in \Omega$,  we need to show
     \be
     \label{pscso1_1}
     \Phi(t, \tau, \omega, x^\pm_\lambda (\tau, \omega ))
     = x^\pm _\lambda (\tau +t, \theta_t \omega ).
     \ee
     By \eqref{seq4}  we find    that, 
    for each  $t \in \R^+$,
     $\tau \in \R$ and $\omega \in \Omega$,
     $$
     x(t+\tau, \tau, \theta_{-\tau} \omega, x^+_\lambda (\tau, \omega))
     $$
     $$
     = 
      {\frac {x^+_\lambda (\tau, \omega)} {
   \sqrt{
     e^{ -2\lambda  t
   - 2 \delta  \omega (t)  } 
   + 2   (x^+_\lambda (\tau, \omega) )^2
    \int_\tau^{t+\tau} e^{2 \lambda (r-t -\tau) + 2\delta (\omega (r-\tau)
   -\omega (t) ) } \beta (r) dr
   }}}
   $$
    $$
     = 
      {\frac {1} {
   \sqrt{
     e^{ -2\lambda  t
   - 2 \delta  \omega (t)  } (x^+_\lambda (\tau, \omega) )^{-2}
   + 2   
    \int_\tau^{t+\tau} e^{2 \lambda (r-t -\tau) + 2\delta (\omega (r-\tau)
   -\omega (t) ) } \beta (r) dr
   }}}
   $$
   $$
     = 
      {\frac {1} {
   \sqrt{
     2e^{ -2\lambda  t
   - 2 \delta  \omega (t)  } 
      \int_{-\infty}^0
    e^{2 \lambda r + 2\delta  \omega (r)
     } \beta (r +\tau) dr
   + 2   
    \int_\tau^{t+\tau} e^{2 \lambda (r-t -\tau) + 2\delta (\omega (r-\tau)
   -\omega (t) ) } \beta (r) dr
   }}}
   $$
   $$
     = 
      {\frac {1} {
   \sqrt{
     2  \int_{-\infty}^0
    e^{2 \lambda ( r-t)  + 2\delta  ( \omega (r) -\omega (t) )
     } \beta (r +\tau) dr
   + 2   
    \int_0^{t} e^{2 \lambda (r-t ) + 2\delta (\omega (r)
   -\omega (t) ) } \beta (r +\tau) dr
   }}}
   $$
   $$
     = 
      {\frac {1} {
   \sqrt{
     2  \int_{-\infty}^t
    e^{2 \lambda ( r-t)  + 2\delta  ( \omega (r) -\omega (t) )
     } \beta (r +\tau) dr
   }}}
   $$
    \be\label{pscso1_2}
     = 
      {\frac {1} {
   \sqrt{
     2  \int_{-\infty}^0
    e^{2 \lambda r + 2\delta  \theta_t \omega (r) 
     } \beta (r +t+ \tau) dr
   }}}
   = x^+_\lambda  (\tau + t, \theta_t \omega).
  \ee
  It follows  from 
  \eqref{phi}   and \eqref{pscso1_2} that
  for each  $t \in \R^+$,
     $\tau \in \R$ and $\omega \in \Omega$,
   \be
     \label{pscso1_3}
     \Phi(t, \tau, \omega, x^+_\lambda (\tau, \omega ))
     = x^+ _\lambda (\tau +t, \theta_t \omega ).
     \ee
     Similarly,  we  can also  verify  that
     for each  $t \in \R^+$,
     $\tau \in \R$ and $\omega \in \Omega$,
   \be
     \label{pscso1_4}
     \Phi(t, \tau, \omega, x^-_\lambda (\tau, \omega ))
     = x^- _\lambda (\tau +t, \theta_t \omega ).
     \ee
     By \eqref{pscso1_3}-\eqref{pscso1_4} we get
     \eqref{pscso1_1}, and hence
     both $x^+_\lambda$   and $x^-_\lambda$
     are random complete quasi-solutions of 
     equation \eqref{geq1}.
     
     We now prove  that
     $x^+_\lambda$   and $x^-_\lambda$
     are tempered.
     Given $c>0$,  $\tau \in \R$
     and $\omega \in \Omega$,
      by \eqref{bcon1}  and
       \eqref{xplus}  we get
     \be\label{pscso1_8}  
     e^{ct} |x^\pm_\lambda (\tau +t, \theta_t \omega ) |
     = 
       {\frac {e^{ct}} {
   \sqrt{
     2  \int_{-\infty}^0
    e^{2 \lambda r + 2\delta  \theta_t \omega (r) 
     } \beta (r +t+ \tau) dr
   }}}
     \le 
       {\frac {e^{ct}} {
   \sqrt{
     2  \beta_0  \int_{-\infty}^0
    e^{2 \lambda r + 2\delta  ( \omega (r+t)  -\omega (t) )
     }  dr
   }}}.
 \ee
 Let $\varepsilon = {\frac 1{2\delta}} \min \{\lambda,  {\frac 12}c \}$.
 By \eqref{asyom} we find that
 for each $\omega \in \Omega$,
  there exists $T = T(\omega)<0$ such 
  that  for all $t \le T$,
  \be\label{pscso1_9}
   \varepsilon t \le  \omega (t)   \le  - \varepsilon t,
 \ee
 which implies   that for all $r \le 0$  and $t \le T$,
 \be\label{pscso1_10}
 \varepsilon r+ \varepsilon t 
  \le
  \omega (r+t)   \le  - \varepsilon r -  \varepsilon t .
 \ee
 It follows   from \eqref{pscso1_9}-\eqref{pscso1_10}
 that,   for all $r \le 0$  and $t \le T$,
 $$
 2 \lambda r + 2\delta  ( \omega (r+t)  -\omega (t) )
 \ge
 2\lambda r + 2\delta \varepsilon r
 + 4\delta \varepsilon t
  \ge
 3\lambda r +c t.
 $$
 Therefore, we get,  for all
 $t \le T$,
\be\label{pscso1_11}
  \int_{-\infty}^0
    e^{2 \lambda r + 2\delta  ( \omega (r+t)  -\omega (t) )
     }  dr
     \ge
      \int_{-\infty}^0
    e^{3 \lambda r + ct }  dr
    \ge {\frac {e^{ct}}{3\lambda} }.
    \ee
    By \eqref{pscso1_8} and \eqref{pscso1_11} we obtain,
     for every  $\tau \in \R$
     and $\omega \in \Omega$,
     $$
     \limsup_{ t \to -\infty}
      e^{ct} |x^\pm_\lambda (\tau +t, \theta_t \omega ) |
    \le  \limsup_{ t \to -\infty}
    \sqrt{ {\frac {3\lambda}{2\beta_0}}} e^{{\frac 12} ct}  =0.
   $$
   This shows    that 
   $x^+_\lambda$   and $x^-_\lambda$
   are tempered.
   Similarly,  by \eqref{pscso1_9} and
   \eqref{pscso1_10},  one  can verify
      ${\frac 1{x^+_\lambda}}$   and
       ${\frac 1{x^-_\lambda}}$
   are  also tempered.
   
   Next,  we prove  that $x^+_\lambda$
   is the only complete quasi-solution
   in $(0, \infty)$ with tempered reciprocal.
   Suppose $\xi$ is an arbitrary complete quasi-solution
   in $(0, \infty)$ such that
   $\xi^{-1}$ is tempered. By definition we have, 
   for all $t \in \R^+$, $\tau \in \R$
   and $\omega \in \Omega$,
   \be\label{pscso1_30}
   \Phi (t, \tau -t, \theta_{-t} \omega,
   \xi(\tau -t, \theta_{-t} \omega ))
   =\xi (\tau, \omega).
   \ee
   On the other hand,   by \eqref{seq5}
   we get, for all $\tau \in \R$  and $\omega \in \Omega$,
   \be\label{pscso1_31}
  x (\tau, \tau -t, \theta_{-\tau} \omega,
   \xi(\tau -t, \theta_{-t} \omega )  )
   =
    {\frac {1} {
   \sqrt{
     e^{ - 2\lambda  t 
   + 2 \delta   \omega (-t) } \xi^{-2}(\tau -t, \theta_{-t} \omega )
   + 2  \int_{\tau -t}^\tau 
    e^{2 \lambda (r-\tau) + 2\delta  \omega (r -\tau)
     } \beta (r) dr
   }}}.
 \ee
 Since $\xi^{-1}$ is tempered, by \eqref{asyom} and \eqref{pscso1_31},
 we obtain,   for 
     all $\tau \in \R$  and $\omega \in \Omega$,
   $$ \lim_{t \to \infty}
    \Phi (t, \tau -t, \theta_{-t} \omega,
   \xi(\tau -t, \theta_{-t} \omega ) )
    =  \lim_{t \to \infty}
    x (\tau, \tau -t, \theta_{-\tau} \omega,
   \xi(\tau -t, \theta_{-t} \omega )  )
    = x^+_\lambda (\tau, \omega),
   $$
   which together with \eqref{pscso1_30}
   gives $\xi(\tau, \omega)
   =x^+_\lambda (\tau, \omega)$, as desired. 
   The uniqueness of $x^-_\lambda$ in $(-\infty, 0)$
   can be proved by  a similar approach.
   The details  are omitted.
    \end{proof}
      
      We now discuss  the stability of 
         the zero solution
      of \eqref{seq1}. 
      
       \begin{lem}
     \label{zerost}
     Suppose  \eqref{bcon1}  holds.
      Then the  zero solution
     of    equation \eqref{seq1} is  pullback asymptotically
     stable in $\R$ if  $\lambda \le 0$; and
     pullback Lyapunov unstable  in $\R$  if    $\lambda>0$.
     \end{lem}
     
     \begin{proof}
     { Case (i)}: $\lambda <0$.
     In this case, we need to prove the asymptotic stability
     of $x=0$.
     Given $\tau \in \R$, $\omega \in \Omega$
     and $\varepsilon>0$,  we  must  find 
     a positive number $\eta =
      \eta(\tau, \omega, \varepsilon)$
      such that for every $t \ge 0$  and
      $x_0 \in (-\eta, \eta)$,  
      \be\label{pzerost1}
      |\Phi (t, \tau -t, \theta_{-t} \omega, x_0) | <
      \varepsilon.
      \ee
      By \eqref{asyom} we see   that there exists
      $T=T(\omega)>0$ such that  for all $t \ge T$,
      \be
      \label{pzerost2}
       {\frac \lambda {2\delta}} t 
      \le \omega (-t) \le 
       -{\frac \lambda {2 \delta} } t .
      \ee
      By \eqref{pzerost2} we get
      \be\label{pzerost3}
      e^{-\lambda t +\delta \omega (-t) }
      \ge e^{-{\frac 12} \lambda t}
      \ge 1, \quad \text{   for all } \ t \ge T.
      \ee
      On the other hand, by   the continuity of
      $\omega$, there exists  a positive number
      $c_0= c_0(\omega)$   such that
      \be\label{pzerost4}
       e^{-\lambda t +\delta \omega (-t) } \ge c_0,
       \quad \text{  for all } t \in [0, T].
       \ee
      Let  $\eta =  
       \min \{  \varepsilon,   \   \varepsilon c_0\}$
      with   $c_0$    as in  \eqref{pzerost4}.
      Then for every $ t \ge 0$  and $x_0 \in  
      (-\eta, \eta)$, 
      it follows   from \eqref{pzerost3}-\eqref{pzerost4}
      that
      $$
      \varepsilon 
       e^{-\lambda t +\delta \omega (-t) } 
       \ge \eta > |x_0|,
       $$
       which  implies   that  
       for every $ t \ge 0$  and $x_0 \in  
      (-\eta, \eta)$,
      \be\label{pzerost9}
     e^{ - 2\lambda  t 
   + 2 \delta   \omega (-t) } x^{-2}_{0}
   + 2  \int_{\tau -t}^\tau 
    e^{2 \lambda (r-\tau) + 2\delta  \omega (r -\tau)
     } \beta (r) dr  >\varepsilon^{-2}.
    \ee
    By \eqref{seq5}  and \eqref{pzerost9}
    we get,   for every $ t \ge 0$  and $x_0 \in  
      (-\eta, \eta)$,
     $$
      |x(\tau, \tau-t, \theta_{-\tau} \omega, x_0)|
      < \varepsilon.
      $$
      Therefore, \eqref{pzerost1}  is satisfied   and thus
      $x=0$ is pullback  Lyapunov stable in  $\R$.
      Note   that for every  $x_0 \in \R$, from
      \eqref{seq5}  and \eqref{pzerost2} we have
      $$
     \limsup_{t \to \infty}
      |x(\tau, \tau-t, \theta_{-\tau} \omega, x_0)|
      \le  \limsup_{t \to \infty}
      e^{\lambda t - \delta \omega (-t)} |x_0| =0,
      $$
      which along with
      \eqref{pzerost1}
      and Definition \ref{fixedpt}
      shows   that $x=0$ is pullback asymptotically stable
      in $\R$
      for $\lambda <0$.
      
       { Case (ii)}: $\lambda =0$.  In this case, by 
       \eqref{bcon1} and
       \eqref{alzero}
       we get   for every $\tau \in \R$   and
       $\omega \in \Omega$,
       $$
       \liminf_{t \to \infty}
       \int_{\tau -t}^\tau
       2e^{2 \delta \omega (r-\tau)}
       \beta (r)  dr
       \ge 
       2 \beta_0 \liminf_{t \to \infty}
        \int_{  -t}^0
       e^{2 \delta \omega (r)}
         dr =\infty,
         $$
         which implies   that  for   given
         $\varepsilon>0$,  there exists
         $T=T(\omega)>0$ such that
         for all $t \ge T$, 
         \be\label{pzerost20}
        2 \int_{\tau -t}^\tau
       e^{2 \delta \omega (r-\tau)}
       \beta (r)  dr > \varepsilon^{-2}.
       \ee
       Let $\eta = \varepsilon c_0$  where
       $c_0$ is the positive number
       in \eqref{pzerost4}   with $\lambda =0$.
       Then for every   $t \in [0. T]$ and $x_0 \in (-\eta, \eta)$, we
       have
       \be\label{pzerost21}
       e^{2\delta \omega (-t) } x_0^{-2}  >  \varepsilon^{-2}.
       \ee
       It follows   from 
       \eqref{pzerost20}-\eqref{pzerost21} that
       for every $t \ge 0$  and $x_0 \in (-\eta, \eta)$,
       $$
       e^{2\delta \omega (-t) } x_0^{-2}
       + 2 \int_{\tau -t}^\tau
       e^{2 \delta \omega (r-\tau)}
       \beta (r)  dr > \varepsilon^{-2},
       $$
       which along   with \eqref{seq5} shows   that
       for every $t \ge 0$  and $x_0 \in (-\eta, \eta)$,
       $$
       |x(\tau, \tau -t, \theta_{-\tau} \omega, x_0 )|
       < \varepsilon.
       $$
       Therefore $x=0$  is pullback Lyapunov stable
       in $\R$
       for $\lambda =0$.
      On the  other   hand, 
      by  \eqref{bcon1} and
       \eqref{alzero} we get,  for each
       $\tau \in \R$    and $\omega \in \Omega$,
       $$
       \limsup_{t \to \infty}
        |x(\tau, \tau -t, \theta_{-\tau} \omega, x_0 )|
        \le
         \limsup_{t \to \infty}
         {\frac {1} {\sqrt{2 \beta_0
          \int^0_{-t} e^{2\delta \omega (r)}  dr}} } =0.
          $$
          This indicates    that $x=0$  is pullback
          asymptotically stable in $\R$   for $\lambda   =0$.
          
          { Case (iii)}: $\lambda >0$. Note  that for every
          $\lambda >0$, $\tau \in \R$
          and $\omega \in \Omega$,
           $x^+_\lambda(\tau, \omega)$  and 
           $x^-_\lambda (\tau , \omega)$
           given by \eqref{xplus} are nonzero.
           In addition, by \eqref{seq6}-\eqref{seq7}, 
           we know     that every
           solution
           $x(\tau, \tau -t, \theta_{-\tau}\omega, x_0)$
           with $x_0 \neq 0$ converges either to
           $x^+_\lambda(\tau, \omega)$ or
           $x^-_\lambda (\tau , \omega)$
           as $t \to \infty$. 
           Therefore, we conclude    that
           the zero solution of \eqref{seq1} is not pullback
           stable in $\R$  for $\lambda>0$.
     \end{proof}
     
       We are now 
       ready to  discuss pitchfork
        bifurcation of random  complete quasi-solutions of 
          \eqref{seq1}.
  
  \begin{thm}
  \label{thms}
   Suppose  \eqref{bcon1}  holds.
     Then the  random complete quasi-solutions of \eqref{seq1}
     undergo a  stochastic pitchfork bifurcation at $\lambda =0$.
     More precisely:
     
         (i) If $\lambda \le 0$,   then $x=0$  is the 
     unique   random complete quasi-solution
     of \eqref{seq1} which is pullback asymptotically stable in $\R$.
     In this case,  the equation has a trivial $\cald$-pullback 
     attractor $\cala = \{ \cala (\tau, \omega) = \{0\}: \tau
     \in \R, \omega \in \Omega \}$.
     
     (ii) If $\lambda>0$,  then the zero solution loses its stability
     and the equation has two more tempered
      random complete quasi-solutions
     $x^+_\lambda>0$  and $x^-_\lambda <0$
     such that 
     \be\label{thms_1}
     \lim_{\lambda \to 0}
     x^\pm_\lambda (\tau, \omega) = 0,
     \quad \text{  for all } \ \tau \in \R
     \ \text{  and }  \omega \in \Omega.
     \ee
     Moreover,   $x^+_\lambda$
    and $x^-_\lambda$  are the only
     complete quasi-solutions
    with  tempered  reciprocals
    in $(0, \infty)$   and $(-\infty, 0)$,
    respectively.
     In this case,  equation \eqref{seq1}
     has a $\cald$-pullback attractor
     $\cala =\{ \cala(\tau, \omega) 
     = [x^-_\lambda (\tau, \omega), 
     x^+_\lambda (\tau, \omega) ]: \tau \in \R,
     \omega \in \Omega \}$,   
     $x^+_\lambda$
     and $x^-_\lambda$
      pullback attracts 
     every compact subset of $(0, \infty)$
     and $(-\infty, 0)$, respectively. 
     \end{thm}
     
     \begin{proof}
     We  first  verify \eqref{thms_1}.
     By
     \eqref{bcon1}, \eqref{alzero}
     and  Fatou\rq{}s lemma we find   that,
     for every $\tau \in \R$  and $\omega
     \in \Omega$,
     $$
     \liminf_{\lambda \to 0}
     \int^0_{-\infty}
     e^{2\lambda r + 2\delta \omega (r)} \beta (r+\tau) dr
     \ge
      \liminf_{\lambda \to 0}
     \int^0_{-\infty}
     \beta_0 e^{2\lambda r + 2\delta \omega (r)}  dr
  \ge 
    \beta_0 \int^0_{-\infty}
      e^{ 2\delta \omega (r)}  dr =\infty,
     $$
     which along with \eqref{xplus}  yields  \eqref{thms_1}.
     Note that the rest of this theorem  is  an immediate
     consequence of 
           \eqref{seq5},
       Lemmas \ref{scso1} and \ref{zerost}. The details
       are omitted here.
     \end{proof}
     
     Next,  we consider pitchfork bifurcation
     of random periodic,   random almost periodic
     and  random almost automorphic solutions
     of \eqref{seq1}.
     Let $\beta: \R \to \R$ be a periodic function with period
     $T>0$. Then by \eqref{xplus} we see that
     for each $\lambda>0$
     and  $\omega \in \Omega$,  both  $x^+_\lambda (\cdot, \omega)$
     and $x^-_\lambda (\cdot, \omega)$ are   $T$-periodic.
     In other words,  
       $x^+_\lambda  $
     and $x^-_\lambda $ are  random  periodic solutions
     of \eqref{seq1}  in this case.   Applying Theorem \ref{thms},
     we  immediately  get  pitchfork bifurcation of 
     random periodic solutions  for \eqref{seq1}.
     In the almost periodic case,  we need   the   following lemma.
     
     \begin{lem}
     \label{apb1}
     Suppose \eqref{bcon1} holds and $\beta: \R
     \to \R$ is almost periodic.
     Then for every $\lambda>0$, 
     the complete quasi-solutions $x^\pm_\lambda$ given
     by \eqref{xplus}  are also almost periodic. 
     \end{lem}
     
      \begin{proof}
      Given $\tau \in \R$  and $\omega \in \Omega$, denote by
      \be\label{papb_1}
      g(\tau, \omega)
      = \int_{-\infty}^0 e^{2\lambda r + 2\delta \omega (r)} \beta (r+\tau)dr.
      \ee
      We  first show  that $g$ given by \eqref{papb_1}
      is a random  almost periodic function.
      Since $\beta$  is almost periodic, given $\varepsilon>0$,  there
      exists $l = l(\omega, \varepsilon) >0$ such that
      every interval of length $l$ contains a $t_0$ such   that
      \be\label{papb_2}
      |\beta(t+ t_0) - \beta (t) |
      < {\frac {\varepsilon} {\int_{-\infty}^0
       e^{2\lambda r + 2\delta \omega (r)} dr} },
       \quad \text{  for  all } \  t \in \R.
       \ee
       By \eqref{papb_1}-\eqref{papb_2} we obtain,
       for all $\tau \in \R$,
       $$
       | g(\tau + t_0, \omega)
       - g(\tau, \omega)|
       \le  \int_{-\infty}^0 e^{2\lambda r + 2\delta \omega (r)} 
       \left | \beta (r+\tau + t_0)
       - \beta (r+ \tau)  \right | dr  
       $$
        \be\label{papb_3}
       <
        {\frac {\varepsilon} {\int_{-\infty}^0
       e^{2\lambda r + 2\delta \omega (r)} dr} }
        \int_{-\infty}^0 e^{2\lambda r + 2\delta \omega (r)} dr 
        \le \varepsilon,
       \ee
        which shows   that
         $g(\cdot, \omega)$ is 
        almost periodic
         for every fixed   $\omega \in \Omega$.  
         By \eqref{bcon1}  we have, for every
         $\tau \in  \R$ and  $\omega \in \Omega$,
         \be\label{papb_4}
         0< \beta_0 
         \int_{-\infty}^0 e^{2\lambda r + 2\delta \omega (r)} dr
         \le g(\tau, \omega)
         \le \beta_1
           \int_{-\infty}^0 e^{2\lambda r + 2\delta \omega (r)} dr.
           \ee
           It follows   from
           \eqref{papb_3}-\eqref{papb_4} that
           for each fixed $\omega \in \Omega$,
           $ g(\cdot, \omega)^{-\frac 12}$ is almost
           periodic.
           Then the almost
           periodicity of $x^\pm(\cdot, \omega)$
           follows     from   \eqref{xplus} immediately,
           and this completes   the proof.
      \end{proof}
      
      Analogously,  for the almost automorphic case,
      we have the following results.
      
      \begin{lem}
     \label{aab1}
     Suppose \eqref{bcon1} holds and $\beta: \R
     \to \R$ is almost  automorphic. 
     Then for every $\lambda>0$, 
     the complete quasi-solutions $x^\pm_\lambda$ given
     by \eqref{xplus}  are also almost  automorphic. 
     \end{lem}
     
     \begin{proof}
     Let $\{\tau_n\}_{n=1}^\infty$ be a sequence of 
     numbers.  Since $\beta$ is almost automorphic, there
     exists a subsequence $\{\tau_{n_m}\}_{m=1}^\infty$
     of $\{\tau_n\}_{n=1}^\infty$     and a function
     $h: \R \to \R$  such that
     for all $t \in \R$,
     \be\label{paab1_1}
     \lim_{m \to \infty}
     \beta (t +\tau_{n_m} ) = h(t)
     \quad \text{  and } \ 
     \lim_{m \to \infty} h(t-  \tau_{n_m} ) = \beta (t).
     \ee
     By \eqref{bcon1}  and \eqref{paab1_1} we have
     \be\label{paab1_2}
     0< \beta_0 \le h(t) \le \beta_1
     \quad \text{  for all } \ t \in \R.
     \ee
     Given $\tau \in \R$  and $\omega \in \Omega$, denote by
     \be\label{paab1_3}
     H(\tau, \omega)
     =  
    {\frac {1} {
   \sqrt{ 
     2   \int_{-\infty}^0
    e^{2 \lambda r + 2\delta  \omega (r)
     } h (r +\tau) dr
     }}}.
     \ee
     Note that the right-hand side of \eqref{paab1_3}
     is well defined due to \eqref{paab1_2}.
     By \eqref{paab1_1}, \eqref{paab1_3}
      and the Lebesgue dominated
     convergence theorem, we get,  for every
     $\tau \in \R$   and $\omega \in \Omega$,
     \be\label{paab1_6}
     \lim_{m \to \infty} x^+_\lambda (\tau + \tau_{n_m}, \omega)
     =\lim_{m \to \infty}
      {\frac {1} {
   \sqrt{ 
     2   \int_{-\infty}^0
    e^{2 \lambda r + 2\delta  \omega (r)
     } \beta  (r +\tau +\tau_{n_m}  ) dr
     }}}
     = 
   H(\tau, \omega),
     \ee
     and
     \be\label{paab1_7}
       \lim_{m \to \infty}  H (\tau - \tau_{n_m}, \omega)
     =\lim_{m \to \infty}
      {\frac {1} {
   \sqrt{ 
     2   \int_{-\infty}^0
    e^{2 \lambda r + 2\delta  \omega (r)
     }  h (r +\tau - \tau_{n_m}  ) dr
     }}}
     = 
  x^+_\lambda (\tau, \omega).
  \ee
  By \eqref{paab1_6}  and \eqref{paab1_7} we find that
  $x^+_\lambda$ is   a random complete quasi-solution
  of \eqref{seq1}.
  By a similar argument, one can verify
  $x^-_\lambda$ is also a random complete solution.
  This completes    the proof.
     \end{proof}

   As a  consequence of Theorem \ref{thms}, 
    Lemmas \ref{apb1}  and \ref{aab1},  we get the
    following pitchfork bifurcation
    of random periodic (almost periodic,  almost automorphic)
    solutions of \eqref{seq1}.
    
   \begin{thm}
  \label{thmsa3}
   Suppose  \eqref{bcon1}  holds and
   $\beta: \R \to \R$ is periodic (almost periodic, almost automorphic).
     Then the  random  periodic 
     (almost periodic, almost automorphic) solutions
     of \eqref{seq1}
     undergo a  stochastic pitchfork bifurcation at $\lambda =0$.
     More precisely:
     
         (i) If $\lambda \le 0$,   then $x=0$  is the 
     unique   random periodic 
     (almost periodic, almost automorphic) solution 
     of \eqref{seq1} which is pullback asymptotically stable
     in $\R$.
     In this case,  the equation has a trivial $\cald$-pullback 
     attractor $\cala = \{ \cala (\tau, \omega) = \{0\}: \tau
     \in \R, \omega \in \Omega \}$.
     
     (ii) If $\lambda>0$,  then the zero solution loses its stability
     and the equation has two more random 
     periodic 
     (almost periodic, almost automorphic) solutions
     $x^+_\lambda>0$  and $x^-_\lambda <0$
     such that 
     \be\label{thmsa3_1}
     \lim_{\lambda \to 0}
     x^\pm_\lambda (\tau, \omega) = 0,
     \quad \text{  for all } \ \tau \in \R
     \ \text{  and }  \omega \in \Omega.
     \ee 
     In this case,  equation \eqref{seq1}
     has a $\cald$-pullback attractor
     $\cala =\{ \cala(\tau, \omega) 
     = [x^-_\lambda (\tau, \omega), 
     x^+_\lambda (\tau, \omega) ]: \tau \in \R,
     \omega \in \Omega \}$.
      Moreover,   
     $x^+_\lambda$
     and $x^-_\lambda$
      pullback attracts 
     every compact subset of $(0, \infty)$
     and $(-\infty, 0)$, respectively. 
     \end{thm}
     
     \begin{proof}
  Since $\beta$ is 
      periodic 
     (almost periodic, almost automorphic),  by 
      Lemmas \ref{apb1}  and \ref{aab1} we know
      that for every $\lambda >0$, 
      the random complete quasi-solutions
      $x^+_\lambda$ and $ x^-_\lambda$
      given by \eqref{xplus} 
      are    periodic 
     (almost periodic, almost automorphic).
     Then,   by Theorem \ref{thms} we
     conclude   the proof.
     \end{proof}

  \subsection{Pitchfork bifurcation of a general 
  non-autonomous stochastic equation  }
  
  In this subsection,  we discuss pitchfork bifurcation
  of the stochastic equation \eqref{geq1}
  with a nonlinearity  $\gamma$ satisfying
  \eqref{gacon1}. 
  We first establish   existence of $\cald$-pullback attractors
  for a generalized  system
  and then construct random complete quasi-solutions.
  The comparison principle will play an important role
  in our arguments.
  
  Given $\tau \in \R$,  consider the   non-autonomous stochastic
  equation defined     for $t>\tau$:
  \be\label{noneq1}
  {\frac {dx}{dt}}
  = f(t,x)   + g(t)  + \delta x \circ {\frac {d\omega}{dt}},
  \quad x(\tau) = x_\tau,
  \ee
  where $\delta >0$,
  $g: \R \to \R$ is a function, $f: \R \times \R
  \to \R$ is a smooth nonlinearity satisfying
  \be\label{f1}
  f(t,0) = 0, \quad f(t,x) x \le -\nu x^2
  + h(t) |x|, \  \text{ for all } \ t ,  \ x \in \R,
  \ee
  for some fixed   $\nu >0$  and 
  $h: \R \to \R$.
  By \eqref{f1} we  see   
  $h(t) \ge 0$    for all    $t \in \R$.
  In the sequel,   we assume  that   $g, h \in
  L^1_{loc} (R)$ and there   exists $\alpha
  \in (0, \nu)$ such that
  \be\label{g1}
  \int_{-\infty}^\tau
  e^{\alpha t}
  ( |g(t)|    + |h(t)| ) dt < \infty,
  \quad   \text{  for all } \ \tau \in \R.
  \ee
  This condition
  will be used   to
  construct pullback absorbing sets
  for \eqref{noneq1}.
  To ensure existence of   
    tempered pullback attractors,  we further require the  
    following  condition for $g$   and $h$:
    for every $c>0$   and $\tau \in \R$,
    \be\label{g2}
    \lim_{s \to -\infty}
    e^{(c-\alpha)s} \int_{-\infty}^{s+\tau}
    e^{\alpha t} 
    ( |g(t) | +  |h(t)| ) dt =0.
    \ee
    Note that
    condition \eqref{g2} is stronger   than
    \eqref{g1},  and both conditions
   do   not require
  $g$ and $h$   to be bounded    as
  $t \to \pm\infty$. 
    Based on \eqref{f1}, we may associate
    a linear   system with \eqref{noneq1}.
    Given $\tau \in \R$   and $y_\tau \in \R$,  consider
     \be\label{leq1}
  {\frac {dy}{dt}}
  = -\nu y   +  |g(t)| + h(t)   + \delta y \circ {\frac {d\omega}{dt}},
  \quad y(\tau) = y_\tau.
  \ee
  By  the comparison principle,
  for every  $\tau \in \R$ and $\omega \in \Omega$,
  if $y_\tau \ge 0$,  then  $y(t, \tau, \omega, y_\tau) \ge 0$
  for all $t \ge \tau$.   This along    with
  \eqref{f1} implies  that the solution 
  $y(t, \tau, \omega, y_\tau)$ of the linear
  equation \eqref{leq1}  is a super-solution of
  \eqref{noneq1} provided    $y_\tau $ is nonnegative.
  Therefore, we  are able    to 
  control solutions of   
  \eqref{noneq1} by \eqref{leq1}
  based on  the comparison principle.
  Given $\tau \in \R$   and $\omega \in \Omega$,   the solution
  of the linear   equation \eqref{leq1}   is given by
  \be\label{sleq1}
  y(t, \tau, \omega, y_\tau)
  =  e^{\nu (\tau -t) -\delta (\omega (\tau) -\omega (t))}
  y_\tau
  + \int_\tau^t e^{\nu (s-t) -\delta (\omega (s)-
  \omega (t) ) } 
  ( |g(s)|   + h(s) ) ds.
  \ee
 Let $\Psi: \R^+ \times \R \times \Omega \times \R$
 $\to \R$ be a mapping given by,
 for every   $t \in \R^+$,
  $\tau \in \R$, 
  $\omega \in \Omega$ and $y_\tau \in \R$,
  \be\label{psil}
  \Psi(t, \tau, \omega,  y_\tau)
  = y(t+\tau, \tau, \theta_{-\tau} \omega, y_\tau).
  \ee
   By  \eqref{sleq1} and 
  \eqref{psil}, one can check   that
  $\Psi$  is   a continuous cocycle on $\R$ over
  $(\Omega, \calf, P,  \{\theta_t \}_{t \in \R})$.
 Next,  we show  $\Psi$   has a unique 
  tempered complete quasi-solution   which pullback attracts
  every tempered sets.

  \begin{lem}
  \label{comsl}
  Suppose  \eqref{g1} and \eqref{g2}  hold.  Then 
  $\Psi$ associated with  
   \eqref{leq1}  has a unique tempered complete
  quasi-solution  $\xi$  given by, 
  for every  $\tau \in \R$   and 
  $\omega \in \Omega$,  
  \be
  \label{comxi}
  \xi(\tau, \omega)
  =  \int_{ -\infty}^0 e^{ \nu s  -\delta \omega (s) }
  ( |g(s +\tau)|   + h(s +\tau) ) ds.
  \ee
  Moreover,  $\Psi$  has a 
  $\cald$-pullback attractor given  by 
  $\cala =\{ \cala (\tau, \omega)
  =\{\xi(\tau, \omega)\}:
  \tau \in \R, \omega \in \Omega \}$.
  If, in addition, $g$   and $h$ are periodic functions
  with period $T>0$, then $\xi$ is also $T$-periodic.
  \end{lem}
  
  \begin{proof}
   First,  by \eqref{asyom}   and \eqref{g1},  
   we can verify      the integral 
   on the right-hand side of \eqref{comxi}
   is well defined    for every
   $\tau \in \R$   and $\omega \in \Omega$.
   If $g$  and $h$ are $T$-periodic, by \eqref{comxi}, we
   see that
    $\xi (\tau +T, \omega) = \xi(\tau, \omega)$
   for all $\tau \in \R$  and $\omega \in \Omega$,
   and hence $\xi$ is $T$-periodic.

   We  now prove $\xi$   is tempered.
   Given    $c_0>0$, $\tau \in \R$   and
   $\omega \in \Omega$, by
    \eqref{comxi}
    we have
     \be\label{pcomsl_1}
   e^{c_0 r} | \xi (\tau +r, \theta_{r} \omega )|
   =
   e^{c_0r} 
 \int_{ -\infty}^0 e^{ \nu s  +\delta \omega (r)  -\delta \omega (s+r)}
  ( |g(s +\tau +r)|   + h(s +\tau +r) ) ds.
  \ee
  Let $\varepsilon = {\frac {1}\delta} \min
    \{ \nu -\alpha, {\frac 14}c_0 \}$.
    By \eqref{asyom}  we find   that
    for every $\omega \in \Omega$,
    there   exists   $T=T(\omega) <0$  such that
    for all  $s \le 0$  and $r \le T$,
    \be\label{pcomsl_2}
    \varepsilon r \le
     \omega (r) \le -  \varepsilon r
     \quad \text { and }
     \  \varepsilon (r+s)  \le
     \omega (r+s) \le -  \varepsilon ( r +s).
     \ee
     It follows  from \eqref{g2}, 
     \eqref{pcomsl_1}-\eqref{pcomsl_2} that
     $$
     \limsup_{r \to -\infty}
     e^{c_0 r} | \xi (\tau +r, \theta_{r} \omega )|
  \le
   \limsup_{r \to -\infty}
   e^{{\frac 12}c_0 r} 
 \int_{ -\infty}^0 e^{ \alpha  s }
  ( |g(s +\tau +r)|   + h(s +\tau +r) ) ds
$$
 $$
 \le e^{-\alpha \tau}
   \limsup_{r \to -\infty}
   e^{ ( {\frac 12}c_0 - \alpha) r} 
 \int_{ -\infty}^{\tau +r}  e^{ \alpha  t }
  ( |g(t)|   + h(t) ) dt
  =0,
 $$
 and hence $\xi$ is tempered. 
 Next, we prove $\xi$   is a  random complete
 quasi-solution of $\Psi$. 
 By \eqref{sleq1}
 and \eqref{comxi}  we find    that for every
 $t \in \R^+$,
 $\tau \in \R$   and
 $\omega \in \Omega$,
 $$
 y(t+\tau, \tau, \theta_{-\tau} \omega, \xi (\tau, \omega))
 =
 e^{-\nu t +\delta \omega (t)} \xi(\tau, \omega)
 +\int_\tau^{t+\tau}
 e^{-\nu (s-t-\tau) -\delta (\omega (s-\tau) -\omega (t))}
 ( |g(s )|   + h(s )) ds
     $$
     $$
     =
    \int_{ -\infty}^0 e^{ \nu (s-t) 
     -\delta ( \omega (s) -\omega (t) ) }
  ( |g(s +\tau)|   + h(s +\tau) ) ds
  +
    \int_{ - t}^0  e^{ -\nu s  -\delta 
    (\omega (s +t)  -\omega (t)) }
  ( |g(s + t+ \tau)|   + h(s +t+ \tau) ) ds.
  $$
    \be\label{pcomsl_3}
     = 
    \int_{ -  \infty }^0  e^{ -\nu s  -\delta 
    (\omega (s +t)  -\omega (t)) }
  ( |g(s + t+ \tau)|   + h(s +t+ \tau) ) ds.
   \ee
   By \eqref{comxi}  and \eqref{pcomsl_3} we get,
   for every
 $t \in \R^+$,
 $\tau \in \R$   and
 $\omega \in \Omega$,
 $$
 \Psi  (t, \tau, \omega, \xi(\tau, \omega))
 = \xi    (\tau +t, \theta_t\omega ).
 $$ 
 This shows   that $\xi$ is a random complete
 quasi-solution of \eqref{leq1}. 
 
 We now prove the attraction property of $\xi$ in $\cald$.
 Recall   that $\cald$ is the collection of
 all  tempered families  given by
 \eqref{tempd}. 
  Let $D  = \{ D(\tau, \omega): \tau \in \R,
  \omega \in \Omega \} \in \cald$ and
  $y_{\tau -t} \in D(\tau -t, \theta_{-t} \omega )$.
   From \eqref{sleq1}
  we have,   for every $t \in \R^+$,
  $\tau \in \R$ and 
  $\omega \in \Omega$,
  $$
   y(\tau,  \tau -t, \theta_{-\tau}\omega, y_{\tau -t} )
  = e^{ -\nu t -\delta \omega (-t)  } y_{\tau -t}
  + \int_{\tau -t}^\tau 
  e^{ \nu (s-\tau) -\delta \omega (s-\tau) }
  ( |g(s)|   + h(s) ) ds
  $$
  \be\label{pcomsl_6}
   = e^{ -\nu t -\delta \omega (-t)  } y_{\tau -t}
  + \int_{ -t}^0 e^{ \nu s  -\delta \omega (s) }
  ( |g(s +\tau)|   + h(s +\tau) ) ds.
  \ee
  By  \eqref{asyom}  we obtain
  $$
  \limsup_{t \to  \infty}
  e^{ -\nu t -\delta \omega (-t)  } |y_{\tau -t}|
  \le
  \limsup_{t \to  \infty}
  e^{ -\nu t -\delta \omega (-t)  } 
  \|D(\tau -t, \theta_{-t} \omega ) \| =0.
  $$ which along with \eqref{psil},
  \eqref{comxi}  and
  \eqref{pcomsl_6}  imply  that
  for every   $D \in \cald$, 
  $\tau \in \R$   and 
  $\omega \in \Omega$,
 \be\label{pcomsl_7}
  \lim_{t \to \infty}
  d(
  \Psi(t,  \tau -t, \theta_{-t} \omega,  
  D(\tau-t, \theta_{-t} \omega)),
  \ \xi(\tau, \omega)) =0.
  \ee
  Note that
  \eqref{pcomsl_7} implies    
  $\xi$ pullback attracts every tempered family of subsets
  of $\R$, and hence
  $\{ \{\xi(\tau, \omega)\}: \tau \in \R,  \omega \in
 \Omega \}$ is a $\cald$-pullback attractor of $\Psi$.
  
  Taking     an arbitrary  tempered  complete
 quasi-solution  $\zeta$ of \eqref{leq1}, we 
 now prove $\zeta =\xi$.  Since $\zeta$ is a
 complete quasi-solution,  we have,
 for each $\tau \in \R$   and
 $ \omega \in \Omega$,
 \be\label{pcomsl_8}
 \Psi(t,  \tau -t, \theta_{-t} \omega,  
  \zeta (\tau-t, \theta_{-t} \omega ))
  = \zeta (\tau, \omega).
  \ee
  Since $\xi$ is tempered, by \eqref{pcomsl_7}
  and \eqref{pcomsl_8} we get
 $\zeta(\tau, \omega) = \xi(\tau, \omega)$
    for every  $\tau \in \R$   and
 $ \omega \in \Omega$.
 This implies  the uniqueness of tempered complete
 quasi-solutions of \eqref{leq1}, and thus 
 completes     the proof.
    \end{proof}
  
  We now prove existence of $\cald$-pullback attractors
  for  equation \eqref{noneq1}.

  \begin{thm}
  \label{nonatt}
  Suppose  \eqref{f1} and \eqref{g1}-\eqref{g2}  hold.  Then 
  $\Phi$ associated with  
   \eqref{noneq1}  has a  unique 
   $\cald$-pullback attractor
   $\cala \in \cald$ which is characterized by,
   for every $\tau \in \R$   and $\omega \in \Omega$,
 \be\label{nonatt_1}
   \cala (\tau, \omega)
     =\{ \xi( \tau, \omega):
   \xi  \text{ is a }   \cald\text{-}{\text{complete  quasi-solution of }} \Phi \}.
   \ee
   \end{thm}
   
   \begin{proof}
   Let  $t \in \R^+$,
    $\tau \in \R$, $\omega \in \Omega$,
   $D \in \cald$   and $x_{\tau -t} \in D(\tau -t, \theta_{-t} \omega)$.
   By  the comparison principle,   we find   the solution $x$   of
   \eqref{noneq1} satisfies
   $$
   |x(\tau, \tau -t, \theta_{-\tau} \omega, x_{\tau -t})|
   \le 
   y(\tau, \tau -t, \theta_{-\tau} \omega, |x_{\tau -t}|),
   $$
   where $y$  is  the solution of the linear   equation
   \eqref{leq1}. 
   Then by \eqref{sleq1}   we have
   \be\label{pnonatt_1}
    |x(\tau, \tau -t, \theta_{-\tau} \omega, x_{\tau -t})|
   \le 
   e^{ -\nu t -\delta \omega (-t)  } |x_{\tau -t}|
  + \int_{ -t}^0 e^{ \nu s  -\delta \omega (s) }
  ( |g(s +\tau)|   + h(s +\tau) ) ds.
  \ee
  Since $x_{\tau -t} \in D(\tau -t, \theta_{-t} \omega)$
  and $D \in \cald$, by \eqref{asyom} we get
 \be\label{pnonatt_2}
  \limsup_{t \to  \infty}
  e^{ -\nu t -\delta \omega (-t)  } |x_{\tau -t}|
  \le
  \limsup_{t \to  \infty}
  e^{ -\nu t -\delta \omega (-t)  } 
  \|D(\tau -t, \theta_{-t} \omega ) \| =0.
  \ee
  It follows    from  \eqref{comxi}  and 
  \eqref{pnonatt_1}-\eqref{pnonatt_2}   that, for 
  every $\tau \in \R$  and $\omega \in \Omega$,
  \be\label{pnonatt_3}
  \limsup_{t \to \infty}
   |x(\tau, \tau -t, \theta_{-\tau} \omega, x_{\tau -t})|
   \le \xi(\tau, \omega),
   \ee
   where $\xi$ is   the complete quasi-solution of
   \eqref{leq1}   given     by \eqref{comxi}.
   On the other hand, by \eqref{pnonatt_1}  and 
   \eqref{pnonatt_2},  there exists $T=T(\tau, \omega, D)>0$
   such that for   all $t \ge T$,
   \be\label{pnonatt_5}
   |\Phi(t, \tau -t, \theta_{-t}  \omega, x_{\tau -t})| = 
   |x(\tau, \tau -t, \theta_{-\tau} \omega, x_{\tau -t})|
   \le 2 \xi(\tau, \omega).
   \ee
   Given $\tau \in \R$ and $\omega \in \Omega$,  define
   $K(\tau , \omega) = [-2\xi(\tau, \omega), 2\xi(\tau, \omega)]$.
   Since $\xi$ is measurable and tempered, by \eqref{pnonatt_5}
   we find $K = \{ K(\tau, \omega): \tau \in \R, \omega \in \Omega \}$
   is a  $\cald$-pullback absorbing set  of $\Phi$.
   Since $K$ is compact,  by Proposition \ref{att}, $\Phi$   has a unique
   $\cald$-pullback attractor $\cala$ which is characterized  by
   \eqref{nonatt_1}.
    \end{proof}
   
   We next further characterize  the structures of the tempered
   attractor of equation \eqref{noneq1}.

  \begin{thm}
  \label{nonatt2}
  Suppose  \eqref{f1} and \eqref{g1}-\eqref{g2}  hold.  Then 
    the cocycle $\Phi$ associated with  
   \eqref{noneq1}  has   two tempered  complete quasi-solutions
   $x^*$  and $x_*$ such that  
   $\cala = \{ [x_*(\tau, \omega), \  x^* (\tau, \omega)],
   \tau \in \R,  \omega \in \Omega \}$
   is the unique $\cald$-pullback attractor of $\Phi$.
  \end{thm}
  
  \begin{proof}
  Let $t_1 \ge t_2 >0$, $\tau \in \R$   and
  $\omega \in \Omega$.  By    the comparison principle,
  for the solution $x$   of \eqref{noneq1}
  we  have
  \be
  \label{natt2_1}
  x(\tau -t_2, \tau -t_1, \theta_{-\tau} \omega,
  \xi(\tau -t_1,  \theta_{-t_1} \omega))
  \le
  y(\tau -t_2, \tau -t_1, \theta_{-\tau} \omega,
  \xi(\tau -t_1,  \theta_{-t_1} \omega)),
  \ee
  where $y$  is the solution  of \eqref{leq1}
  and $\xi$   is   given by \eqref{comxi}. 
  Since $\xi$    is  a complete quasi-solution
  of \eqref{leq1}, we get
  \be\label{natt2_2}
   y(\tau -t_2, \tau -t_1, \theta_{-\tau} \omega,
  \xi(\tau -t_1,  \theta_{-t_1} \omega))
  =\Psi (t_1-t_2, \tau-t_1,
  \theta_{-t_1} \omega,
   \xi(\tau -t_1,  \theta_{-t_1} \omega ))
   = \xi(\tau -t_2,  \theta_{-t_2} \omega ).
   \ee
   By  \eqref{natt2_1}-\eqref{natt2_2}  we obtain
    \be
  \label{natt2_3}
  x(\tau -t_2, \tau -t_1, \theta_{-\tau} \omega,
  \xi(\tau -t_1,  \theta_{-t_1} \omega))
  \le \xi(\tau -t_2,  \theta_{-t_2} \omega ).
  \ee
  By \eqref{natt2_3} 
  and  the comparison principle,  we get
 $$
x(\tau, \tau-t_2, \theta_{-\tau} \omega,
  x(\tau -t_2, \tau -t_1, \theta_{-\tau} \omega,
  \xi(\tau -t_1,  \theta_{-t_1} \omega) ))
  \le 
  x(\tau, \tau-t_2, \theta_{-\tau} \omega,
  \xi(\tau -t_2,  \theta_{-t_2} \omega  )),
 $$
 which implies that 
 for all $t_1 \ge t_2 >0$, $\tau \in \R$   and
  $\omega \in \Omega$, 
 \be\label{natt2_6}
 x(\tau, \tau-t_1, \theta_{-\tau} \omega,
  \xi(\tau -t_1,  \theta_{-t_1} \omega ))
  \le 
  x(\tau, \tau-t_2, \theta_{-\tau} \omega,
  \xi(\tau -t_2,  \theta_{-t_2} \omega  )).
\ee
  By \eqref{natt2_6}   we find that
  $  x(\tau, \tau-t, \theta_{-\tau} \omega,
  \xi(\tau -t,  \theta_{-t} \omega )) $
  is monotone in $t \in \R^+$  for each
  fixed $\tau  $   and $\omega  $.
  Since $\xi$  is tempered, by \eqref{pnonatt_5}
  we  see   $  x(\tau, \tau-t, \theta_{-\tau} \omega,
  \xi(\tau -t,  \theta_{-t} \omega )) $ is bounded
  in $t\in \R^+$.  Therefore, for every $\tau \in \R$   and
  $\omega \in \Omega$, 
  there   exists $x^*(\tau, \omega)\in \R$ 
    such that
  \be
  \label{natt2_7}
  \lim_{t \to \infty}
     x(\tau, \tau-t, \theta_{-\tau} \omega,
  \xi(\tau -t,  \theta_{-t} \omega ))
  = x^*(\tau, \omega).
  \ee
  By  the attraction property of $\cala$ of the $\cald$-pullback
  attractor of \eqref{noneq1}, we have
  $x^*(\tau, \omega) \in \cala(\tau, \omega)$
  for every $\tau$  and $\omega$. 
  By \eqref{pnonatt_3}  and \eqref{natt2_7}
  we get
  $ | x^*(\tau, \omega)|  \le \xi(\tau, \omega)$,
  and hence $x^*$ is tempered.
  By a similar argument,  we can show  there exists
    $x_*(\tau, \omega)
    \in \cala(\tau, \omega)$
    with $ |x_*(\tau, \omega)|   \le  \xi(\tau, \omega)$ 
    such that
  \be
  \label{natt2_8}
  \lim_{t \to \infty}
     x(\tau, \tau-t, \theta_{-\tau} \omega,
  -  \xi(\tau -t,  \theta_{-t} \omega ))
  = x_*(\tau, \omega).
  \ee
  Note that  $x_*$ is tempered.
  By  \eqref{natt2_7}-\eqref{natt2_8} and   the
  comparison principle, we have
  $x_*(\tau, \omega) \le x^*(\tau, \omega)$.
  Note that \eqref{pnonatt_3} implies
  \be\label{natt2_9}
  \cala(\tau, \omega) \subseteq [-\xi(\tau, \omega),
  \xi(\tau, \omega)],
  \quad \text{ for all } \ \tau \in \R
  \ \text{  and }
  \ \omega \in \Omega.
  \ee
   Based on \eqref{natt2_9} we will  prove
   \be\label{natt2_10}
  \cala(\tau, \omega) \subseteq [x_*(\tau, \omega),
   x^*(\tau, \omega)],
  \quad \text{ for all } \ \tau \in \R
  \ \text{  and }
  \ \omega \in \Omega.
  \ee
  Let $x_0 \in \cala(\tau, \omega)$ and $t_n \to \infty$.
  By  the invariance of $\cala$,  
  there exists
  $x_{0,n} \in \cala(\tau -t_n, \theta_{-t_n} \omega )$
  for every $n $ 
  such that
 $ x_0 =  x(\tau, \tau-t_n, \theta_{-\tau} \omega, x_{0,n})$.
  Since  $x_{0,n} \in \cala(\tau -t_n, \theta_{-t_n} \omega )$,
  by \eqref{natt2_9} we  have
  $|x_{0,n}| \le \xi (\tau -t_n, \theta_{-t_n} \omega )$.
  Then by  the comparison principle we get
  $$
   x_0=x(\tau, \tau-t_n, \theta_{-\tau} \omega, x_{0,n})
   \le 
    x(\tau, \tau-t_n, \theta_{-\tau} \omega,
       \xi (\tau -t_n, \theta_{-t_n} \omega  )).
       $$
       Letting  $n \to \infty$, by \eqref{natt2_7} we get
       $x_0 \le x^*(\tau, \omega)$. 
       Similarly,  by \eqref{natt2_8} one can verify
        $x_0  \ge  x_*(\tau, \omega)$. 
        Thus \eqref{natt2_10}   follows. 
        Before proving   the converse of \eqref{natt2_10}, we
        first prove $x^*$  and $x_*$  are complete quasi-solutions of
        \eqref{noneq1}. 
        By \eqref{natt2_7}  and the continuity of solutions in initial
        data,   we get for every
        $s \in \R^+$, $\tau \in \R$   and $\omega \in \Omega$,
        $$ x(s+ \tau, \tau, \theta_{-\tau} \omega, x^*(\tau, \omega))
   = \lim_{t \to \infty}
      x(s+ \tau, \tau, \theta_{-\tau} \omega,    
         x(\tau, \tau-t, \theta_{-\tau} \omega,
  \xi(\tau -t,  \theta_{-t} \omega ) ))
$$
$$  = \lim_{t \to \infty}   
         x(s+\tau, \tau-t, \theta_{-\tau} \omega,
  \xi(\tau -t,  \theta_{-t} \omega ))
  = \lim_{r \to \infty}   
         x(s+\tau,  s+\tau-r, \theta_{-\tau} \omega,
  \xi(s+ \tau -r,  \theta_{s-r} \omega ))
  $$
  \be\label{natt2_20}
   = \lim_{t \to \infty}   
         x(s+\tau,  s+\tau-t, \theta_{-\tau-s} \theta_s \omega,
  \xi(s+ \tau -t,  \theta_{-t}  \theta_{s} \omega ) )
  = x^*(\tau +s, \theta_s \omega),
 \ee
  where the last limit is obtained by \eqref{natt2_7}. 
  By \eqref{natt2_20} we get 
 for every
        $s \in \R^+$, $\tau \in \R$   and $\omega \in \Omega$,
       \be\label{natt2_21}
       \Phi(s, \tau, \omega, x^*(\tau, \omega))
       =   x^*(\tau +s, \theta_s \omega),
       \ee
       and hence $x^*$   is a complete quasi-solution
       of $\Phi$. 
       Similarly, one can check  that $x_*$ is a complete
       quasi-solution of $\Phi$, i.e., 
        for every
        $s \in \R^+$, $\tau \in \R$   and $\omega \in \Omega$,
       \be\label{natt2_22}
       \Phi(s, \tau, \omega, x_*(\tau, \omega))
       =   x_*(\tau +s, \theta_s \omega).
       \ee
       Finally, we prove    the converse of 
       \eqref{natt2_10}, i.e., 
       \be\label{natt2_25}
    [x_*(\tau, \omega),
   x^*(\tau, \omega)] \subseteq
    \cala(\tau, \omega)    ,
  \quad \text{ for all } \ \tau \in \R
  \ \text{  and }
  \ \omega \in \Omega.
  \ee
  Given $\tau \in \R$,
  $\omega \in \Omega$
  and $z \in [x_*(\tau, \omega), x^*(\tau, \omega)]$,
  by  the comparison principle,
  we  find that
  $x(t+\tau, \tau, \theta_{-\tau} \omega,  z)$ 
  is defined  for   all $t \in \R$
  and 
  $$
  x(t+\tau, \tau, \theta_{-\tau} \omega,  x_* (\tau, \omega) )
  \le
  x(t+\tau, \tau, \theta_{-\tau} \omega,  z)
  \le 
    x(t+\tau, \tau, \theta_{-\tau} \omega,  x^* (\tau, \omega) ),
    $$
    that is,
      $$
      \Phi(t, \tau, \omega, x_* (\tau, \omega)  )
  \le
  x(t+\tau, \tau, \theta_{-\tau} \omega,  z)
  \le 
     \Phi(t, \tau, \omega, x^* (\tau, \omega)  ).
    $$
    This along with \eqref{natt2_21}-\eqref{natt2_22} shows  that
    \be\label{natt2_30}
       x_* (\tau +t,  \theta_t \omega )
  \le
  x(t+\tau, \tau, \theta_{-\tau} \omega,  z)
  \le 
      x^* (\tau +t,  \theta_t \omega  ).
\ee
Since $x^*$   and $x_*$ are tempered, by \eqref{natt2_30} we
know  that 
$\psi (t, \tau, \omega)
= x(t+\tau, \tau, \theta_{-\tau} \omega,  z)$
is a $\cald$-complete solution of $\Phi$.
Therefore, by  \eqref{nonatt_1} we find that
$z = \psi (0, \tau, \omega) \in \cala(\tau, \omega)$,
which yields \eqref{natt2_25}.
It follows    from 
           \eqref{natt2_10}
           and        \eqref{natt2_25} that
              \be\label{natt2_40}
\cala(\tau, \omega)  =  [x_*(\tau, \omega),
   x^*(\tau, \omega)]  ,
  \quad \text{ for all } \ \tau \in \R
  \ \text{  and }
  \ \omega \in \Omega.
  \ee
  By \eqref{natt2_21}-\eqref{natt2_22}
  and \eqref{natt2_40} we conclude    the proof.
   \end{proof}
   
   In what   follows,  we discuss pitchfork bifurcation
   of complete quasi-solutions of 
   of \eqref{geq1} as $\lambda$ crosses zero from below.
   Let
   \be\label{f2}
   f(t,x) =  \lambda x -\beta (t) x^3
   +\gamma (t,x),
   \quad t  \in \R
   \text{  and } \ x \in \R.
 \ee
  By   \eqref{bcon1}   and  \eqref{gacon1} we have,
  for all $t\in \R$  and $x \in \R$,
 \be\label{nf1}
    f(t,x) x
  \le \lambda x^2
   -\beta_0 x^4 + c_2 x^4
   \le -x^2
    + \left ( (\lambda +1) |x|
    -(\beta_0 -c_2) |x|^3  \right ) |x|.
  \ee
 Since $\beta_0 > c_2$,
  by Young\rq{}s inequality,  there exists   a
  positive number $c$   such that
  $$
  |\lambda +1 |  |x|
  \le {\frac 12} (\beta_0 -c_2) |x|^3 +c,
  $$
  which along with \eqref{nf1} implies   that
  for all $t \in \R$  and $x\in \R$,
  $$
    f(t,x) x
  \le -x^2 + c |x|.
  $$
  Therefore, $f$ 
  given by
  \eqref{f2}  satisfies 
  condition \eqref{f1}
  with $\nu =1$   and $h(t) =c$
  for all $t \in \R$.  Let $g(t) =0$ for all
  $t \in \R$. Then $g$  and $h$   satisfy
  \eqref{g1}  and \eqref{g2}
  for every $\alpha>0$. 
  In this case,  $\xi$ as defined  by \eqref{comxi}
  becomes 
  \be\label{gcomxi}
  \xi (\tau, \omega)
  = c \int_{-\infty}^0 e^{s -\delta \omega (s)} ds,
  \quad \tau \in \R  \ \text{  and } \  \omega \in \Omega.
  \ee
  It is easy to check that
  $\xi$ given by \eqref{gcomxi}  has a tempered
  reciprocal, i.e.,   for every $c_0>0$, $\tau \in \R$
  and $\omega \in \Omega$,
  \be\label{grecomxi}
  \lim_{t \to \infty}
  e^{-c_0 t} \xi^{-1}(\tau -t, \theta_{-t} \omega)
  =0.
  \ee
   By 
  Theorem \ref{nonatt2} we find   that
  for each   $\lambda \in \R$, 
  equation \eqref{geq1} has a unique
   $\cald$-pullback attractor
  $\cala_\lambda \in \cald$ such  that
  for every $\tau \in \R$  and
  $\omega \in \Omega$,
\be \label{glamatt}
  \cala_\lambda (\tau, \omega)
  = [x^-_\lambda (\tau, \omega), \ 
  x^+_\lambda (\tau,  \omega) ],
 \ee
  where $x^+_\lambda$
  and $x^-_\lambda$  are tempered
  complete quasi-solutions   of 
  \eqref{geq1} given by 
   \be
  \label{glamatt_1}
    x^+_\lambda(\tau, \omega)
    =
  \lim_{t \to \infty}
     x(\tau, \tau-t, \theta_{-\tau} \omega,
  \xi(\tau -t,  \theta_{-t} \omega )),
  \ee
  and 
  \be
  \label{glamatt_2}
  x^-_\lambda (\tau, \omega)
  =
  \lim_{t \to \infty}
     x(\tau, \tau-t, \theta_{-\tau} \omega,
  - \xi(\tau -t,  \theta_{-t} \omega )),
  \ee
  with  $\xi$  being   defined  by 
  \eqref{gcomxi}. 
  Note that \eqref{glamatt_1}  and \eqref{glamatt_2}
  follow  from 
    \eqref{natt2_7} and \eqref{natt2_8}
    by replacing $x^*$
    and $x_*$
    by $x^+_\lambda$
    and $x^-_\lambda$, respectively.
    By the comparison principle, 
    we find   from \eqref{glamatt_1}-\eqref{glamatt_2}
    that  $x^+_\lambda(\tau, \omega) \ge 0$
    and $x^-_\lambda(\tau, \omega) \le 0$
    for all $\tau \in \R$  and  $\omega \in \Omega$.
    Actually,  $x^+_\lambda(\tau, \omega) > 0$
    and $x^-_\lambda(\tau, \omega) < 0$
    as demonstrated below.

  \begin{lem}
  \label{gcomslam}
  Suppose \eqref{bcon1}  and \eqref{gacon1} hold.
   Then for every $\lambda \in \R$,  the 
   tempered complete quasi-solutions
   $x^+_\lambda$  and $x^-_\lambda$ in \eqref{glamatt}
   satisfy, for  all $\tau \in \R$   and
   $\omega \in \Omega$,
    $x^+_\lambda (\tau, \omega) >0$,
    $x^-_\lambda (\tau, \omega) <0$ and 
   \be\label{gcomslam1}
     {\frac 1{ \sqrt{
       2 (\beta_1 -c_1)  \int_{-\infty}^0
      e^{2\lambda  r + 2\delta \omega r
       } dr
      }} }
       \le
   |x_\lambda ^\pm (\tau, \omega)|
   \le
     {\frac 1{
   \sqrt{2(\beta_0 -c_2) \int^0_{-\infty}
   e^{2 \lambda r + 2 \delta \omega (r)} dr
     } } }.
     \ee
     If $\lambda>0$,  then   the zero solution
     of \eqref{geq1}  is unstable in $\R$.
     \end{lem}
     
     \begin{proof}
     For $x>0$ we introduce a new variable
     $z = x^{-2}$. By \eqref{geq1} we find  
     that $z$  satisfies,
     \be\label{pgcslam_1}
     {\frac {dz}{dt}}
     =-2 \lambda z
     + 2 \beta (t)
     -2 z^{\frac 32} \gamma (t, z^{-\frac 12})
     -2 \delta z \circ {\frac {d \omega}{dt}},
      \quad  z(\tau) =  z_\tau.
     \ee
       By \eqref{gacon1} we have
     \be\label{pgcslam_2}
     -2 c_2 \le -2 z^{\frac 32} \gamma (t, z^{-\frac  12})
     \le -2c_1 , \quad \text{ for all}  \  z >0.
     \ee
     Consider the linear equations
     for $t>\tau$   with $\tau \in \R$,
     \be\label{pgcslam_3}
      {\frac {du}{dt}}
     =-2 \lambda u
     + 2 \beta (t)
     -2c_1  
     -2 \delta u  \circ {\frac {d \omega}{dt}},
     \quad u(\tau) =  u_\tau,
     \ee
     and
      \be\label{pgcslam_4}
      {\frac {dv}{dt}}
     =-2 \lambda  v
     + 2 \beta (t)
     -2c_2  
     -2 \delta v  \circ {\frac {d \omega}{dt}},
     \quad v(\tau) =  v_\tau.
     \ee
     Given   $\tau \in \R$ and $\omega \in \Omega$,
     the solutions $u$  and $v$ of \eqref{pgcslam_3}
     and \eqref{pgcslam_4}  are given by
    $$
     u(t, \tau, \omega, u_\tau)
     = e^{2\lambda (\tau -t)  + 2\delta (\omega (\tau) -
     \omega (t) )} u_\tau
     + 2 \int_\tau^t e^{2\lambda (r-t) + 2\delta (\omega (r)
     - \omega (t)) }
     (\beta (r) -c_1 ) dr,
   $$
     and
     $$
     v(t, \tau, \omega, v_\tau)
     = e^{2\lambda (\tau -t)  + 2\delta (\omega (\tau) -
     \omega (t) )} v_\tau
     + 2 \int_\tau^t e^{2\lambda (r-t) + 2\delta (\omega (r)
     - \omega (t)) }
     (\beta (r) -c_2 ) dr.
    $$
     Therefore,  for every
     $t \in \R^+$,    $\tau \in \R$ and $\omega \in \Omega$,
     we have
      \be\label{pgcslam_7}
     u(\tau, \tau-t,  \theta_{-\tau} \omega,  u_{\tau - t} )
     = e^{ -2\lambda  t  + 2\delta 
     \omega (-t) } u_{\tau -t}
     + 2 \int_{\tau-t}^\tau e^{2\lambda (r-\tau) + 2\delta \omega (r-\tau)
       }
     (\beta (r) -c_1 ) dr,
   \ee
   and
    \be\label{pgcslam_8}
     v(\tau, \tau-t,  \theta_{-\tau} \omega,  v_{\tau -t} )
     = e^{ -2\lambda  t  + 2\delta 
     \omega (-t) } v_{\tau -t}
     + 2 \int_{\tau-t}^\tau e^{2\lambda (r-\tau) + 2\delta \omega (r-\tau)
       }
     (\beta (r) -c_2 ) dr.
   \ee
      By \eqref{pgcslam_2} we see that
     $u$  and $v$ are super- and sub-solutions of
     \eqref{pgcslam_1}, respectively.
     Since  $x = z^{-\frac 12}$ for $x >0$,  we get,
      for every
     $t \in \R^+$,    $\tau \in \R$,  $\omega \in \Omega$
     and $x_{\tau -t} >  0$,
      \be\label{pgcslam_9}
     {\frac 1{ \sqrt{
     u(\tau, \tau-t,  \theta_{-\tau} \omega,  x^{-2}_{\tau - t} )
     }} }
     \le x(\tau, \tau-t,  \theta_{-\tau} \omega,  x_{\tau - t} )
     \le
       {\frac 1{ \sqrt{
     v(\tau, \tau-t,  \theta_{-\tau} \omega,  x^{-2}_{\tau - t} )
     }} }.
     \ee
     Similarly,  for $x_{\tau -t} <  0$, one can verify   that
     $-x$ satisfies \eqref{pgcslam_9}.   So 
     for every
     $t \in \R^+$,    $\tau \in \R$,  $\omega \in \Omega$
     and $x_{\tau -t}  \neq 0 $, we have
       \be\label{pgcslam_10}
     {\frac 1{ \sqrt{
     u \left  (\tau, \tau-t,  \theta_{-\tau} \omega,    
      x^{-2}_{\tau -t} 
     \right )
     }} }
     \le  | x(\tau, \tau-t,  \theta_{-\tau} \omega, 
        x_{\tau -t}  ) |
     \le
       {\frac 1{ \sqrt{
     v \left (\tau, \tau-t,  \theta_{-\tau} \omega,  
       x^{-2}_{\tau -t} 
     \right  )
     }} },
\ee
from which we get, 
    for all   $\tau \in \R$
    and  $\omega \in \Omega$, 
           \be\label{pgcslam_11}
            {\frac 1{ \sqrt{
     u \left  (\tau, \tau-t,  \theta_{-\tau} \omega,    
       \ \xi^{-2} (\tau - t,  \theta_{-t} \omega   )
     \right )
     }} }
     \le  | x(\tau, \tau-t,  \theta_{-\tau} \omega, 
          \pm  \xi (\tau - t,  \theta_{-t} \omega   ) ) |
       $$
       $$
     \le
       {\frac 1{ \sqrt{
     v \left (\tau, \tau-t,  \theta_{-\tau} \omega,  
         \xi^{-2} (\tau - t,  \theta_{-t} \omega   )
     \right  )
     }} },
     \ee
     where $\xi$ is given by 
     \eqref{gcomxi}.
     Letting $t \to \infty$,  by
     \eqref{asyom}, 
      \eqref{grecomxi},
      \eqref{glamatt_1}-\eqref{glamatt_2}
     and \eqref{pgcslam_7}-\eqref{pgcslam_8} we obtain
     from \eqref{pgcslam_11}  that,   
     for all
      $\tau \in \R$  and $\omega \in \Omega$, 
      $$
         {\frac 1{ \sqrt{
       2 \int_{-\infty}^0
      e^{2\lambda  r + 2\delta \omega r
       }
     (\beta (r+\tau) -c_1 ) dr
      }} }
       \le
       |  x^\pm_\lambda (\tau ,    \omega   ) |
    \le 
       {\frac 1{ \sqrt{
       2 \int_{-\infty}^0
      e^{2\lambda  r + 2\delta \omega r
       }
     (\beta (r+\tau) -c_2 ) dr
      }} }.
     $$
     Therefore,  by \eqref{bcon1} we have,   for all
      $\tau \in \R$  and $\omega \in \Omega$,  
      \be\label{pgcslam_20}
         {\frac 1{ \sqrt{
       2 (\beta_1 -c_1)  \int_{-\infty}^0
      e^{2\lambda  r + 2\delta \omega r
       } dr
      }} }
       \le
       |  x^\pm_\lambda (\tau ,    \omega   ) |
    \le 
       {\frac 1{ \sqrt{
       2(\beta_0 -c_2)  \int_{-\infty}^0
      e^{2\lambda  r + 2\delta \omega r
       }  dr
      }} },
     \ee
     which implies $x^+_\lambda (\tau, \omega)>0$
     and $x^-_\lambda (\tau, \omega)<0$.
     On the other hand,  for every $x_0 \neq 0$,
  by \eqref{pgcslam_7}
  and \eqref{pgcslam_10} we get
  \be\label{pgcslam_21}
  \liminf_{t \to \infty}
   | x(\tau, \tau-t,  \theta_{-\tau} \omega, 
        x_{0}  ) | 
    \ge
     {\frac 1{ \sqrt{
       2(\beta_1 -c_1)  \int_{-\infty}^0
      e^{2\lambda  r + 2\delta \omega r
       } dr
      }} }.
     \ee
     By \eqref{pgcslam_21},    
     the zero solution of \eqref{geq1} is unstable
     in $\R$.
     Thus,  by \eqref{pgcslam_20} we conclude   the
     proof.
        \end{proof}

           We   now  present 
        pitchfork
        bifurcation of random  complete quasi-solutions of 
          \eqref{geq1}.
  
  \begin{thm}
  \label{thmgs}
   Suppose  \eqref{bcon1} and \eqref{gacon1}  hold.
     Then the  random complete quasi-solutions of \eqref{geq1}
     undergo a  stochastic pitchfork bifurcation at $\lambda =0$.
     More precisely:
     
         (i) If $\lambda \le 0$,   then $x=0$  is the 
     unique   random complete quasi-solution
     of \eqref{seq1} which is pullback asymptotically stable
     in $\R$.
     In this case,  the equation has a trivial $\cald$-pullback 
     attractor $\cala_\lambda = \{ \cala_\lambda (\tau, \omega) = \{0\}: \tau
     \in \R, \omega \in \Omega \}$.
     
     (ii) If $\lambda>0$,  then the zero solution loses its stability
     and the equation has two more  tempered
     random complete quasi-solutions
     $x^+_\lambda >  0$  and $x^-_\lambda  < 0$
     such that 
    \be\label{thmgs_1}
     \lim_{\lambda \to 0}
     x^\pm_\lambda (\tau, \omega) = 0,
     \quad \text{  for all } \ \tau \in \R
     \ \text{  and }  \omega \in \Omega.
     \ee
     In this case,  equation \eqref{seq1}
     has a $\cald$-pullback attractor
     $\cala_\lambda =\{ \cala_\lambda(\tau, \omega) 
     = [x^-_\lambda (\tau, \omega), 
     x^+_\lambda (\tau, \omega) ]: \tau \in \R,
     \omega \in \Omega \}$. 
     \end{thm}
     
     \begin{proof}
      (i)  If  $\lambda \le 0$,  by
      \eqref{alzero} and  \eqref{gcomslam1} we have,
      for every $\tau \in \R$   and $\omega \in \Omega$,
  \be\label{pthmgs_1}
   |x_\lambda ^\pm (\tau, \omega)|
   \le
     {\frac 1{
   \sqrt{2(\beta_0 -c_2) \int^0_{-\infty}
   e^{ 2 \delta \omega (r)} dr
     } } }
     = 0,
   \ee
    and hence $x_\lambda ^\pm (\tau, \omega)=0$.
    In this case, by \eqref{glamatt} we see that
     zero  is the 
     only   complete quasi-solution
     of \eqref{geq1} which is pullback asymptotically stable.
     In addition,    $ \cala (\tau, \omega) = \{0\}$
     for all  $ \tau
     \in \R$  and  $\omega \in \Omega  $.
     
     (ii) If $\lambda>0$, by Lemma \ref{gcomslam} we
     know that $x=0$ is unstable. Moreover, by 
     \eqref{gcomslam1}, \eqref{pthmgs_1} 
     and Fatou's lemma, we have
      $$
   \limsup_{\lambda \to  0}
    |x_\lambda ^\pm (\tau, \omega)|
   \le \limsup_{\lambda \to  0}
     {\frac 1{
   \sqrt{2(\beta_0 -c_2) \int^0_{-\infty}
   e^{ 2\lambda r+ 2 \delta \omega (r)} dr
     } } }
     \le
     {\frac 1{
   \sqrt{2(\beta_0 -c_2) \int^0_{-\infty}
   e^{ 2 \delta \omega (r)} dr
     } } }
     = 0,
    $$ 
    which  implies  \eqref{thmgs_1}
    and  thus  
    completes   the proof.
     \end{proof}

         As a consequence of Theorem \ref{thmgs},
         we have   the following     pitchfork
        bifurcation of random   periodic solutions of 
          \eqref{geq1}.
  
  \begin{thm}
  \label{thmgps}
  Let $T$ be a positive number such that
   $\beta(t+T) =\beta (t)$
  and $\gamma (t+T, x) = \gamma (t,x)$
  for all $t \in \R$  and $x\in \R$.
  If   \eqref{bcon1} and \eqref{gacon1}  hold,
    then    random periodic solutions of \eqref{geq1}
     undergo a  stochastic pitchfork bifurcation at $\lambda =0$.
     More precisely:
     
         (i) If $\lambda \le 0$,   then $x=0$  is the 
     unique   random  periodic solution
     of \eqref{seq1} which is pullback asymptotically stable
     in $\R$.
     In this case,  the equation has a trivial $\cald$-pullback 
     attractor.
     
     (ii) If $\lambda>0$,  then the zero solution loses its stability
     and the equation has two more random periodic solutions
     $x^+_\lambda > 0$  and $x^-_\lambda  < 0$
     such that 
    \be\label{thmgps_1}
     \lim_{\lambda \to 0}
     x^\pm_\lambda (\tau, \omega) = 0,
     \quad \text{  for all } \ \tau \in \R
     \ \text{  and }  \omega \in \Omega.
     \ee
     In this case,  equation \eqref{seq1}
     has a $\cald$-pullback attractor
     $\cala =\{  [x^-_\lambda (\tau, \omega), 
     x^+_\lambda (\tau, \omega) ]: \tau \in \R,
     \omega \in \Omega \}$. 
     \end{thm}

     \begin{proof}
     By Theorem \ref{thmgs}, we only need  to
     show that  for  each  $\lambda>0$, 
      the tempered complete quasi-solutions
      $x^+_\lambda$  and $x^-_\lambda$
     in \eqref{glamatt} are $T$-periodic.
     Note that 
       $x^+_\lambda$  and $x^-_\lambda$
      are defined by
     \eqref{natt2_7} and \eqref{natt2_8}
     with $x^*$  and $x_*$ being replaced by
     $x^+_\lambda$  and $x^-_\lambda$, respectively.
     In the present case,
          by Lemma \ref{comsl}  we find   that
          $\xi$   given   by 
     \eqref{comxi} is $T$-periodic.
      Then, by  \eqref{natt2_7}
     and   the periodicity  of $\beta$ and $\gamma$, 
       we get
     for all  $\tau\in \R$  and $\omega \in \Omega$,
     $$
      x^+_\lambda (\tau +T, \omega)
     = \lim_{t \to \infty}
     x(\tau + T, \tau +T -t,  \theta_{-\tau -T} \omega,
     \xi(\tau +T -t, \theta_{-t} \omega ))
     $$
     $$
      = \lim_{t \to \infty}
     x(\tau, \tau  -t,  \theta_{-\tau } \omega,
     \xi(\tau  -t, \theta_{-t} \omega )) = x^+_\lambda (\tau , \omega),
     $$
     which shows   that $ x^+_\lambda$ is $T$-periodic.
     Similarly, one can verify that 
     $x^-_\lambda $ is also $T$-periodic.
     The details  are omitted.
        \end{proof}

  \section{Transcritical  bifurcation of  stochastic equations  }
\setcounter{equation}{0}
 
 In this section,  we discuss
   transcritical    bifurcation
 of  the   one-dimensional
 non-autonomous 
   stochastic equation
 given by
 \be
 \label{tgeq1}
 {\frac {dx}{dt}} =\lambda x -\beta (t) x^2
 + \gamma (t,x) + \delta x  \circ {\frac {d\omega}{dt}},
 \quad  x(\tau) = x_\tau , \quad t>\tau,
 \ee
 where  
   $\lambda$,  $\delta$  and 
     $\beta$ are    the same as in     \eqref{geq1};
     particularly,   $\beta$  satisfies
     \eqref{bcon1}.
     However, in the present case, we assume the smooth
     function $\gamma$ satisfies    the following condition:
  there exist two nonnegative numbers
 $c_1$  and $c_2$   with $c_1 \le c_2 <\beta_0$
 such that
\be\label{gacon2}
 c_1x^2  \le  \gamma (t,x) \le c_2 x^2
 \quad \text{  for all } \  t \in \R
 \quad  \text{ and } \  x \in \R.
 \ee
By \eqref{gacon2} we have  
 $\gamma (t,0) = 0$
 for   all   $t \in  \R$, and hence 
 $x=0$ is a  fixed point of  
 \eqref{tgeq1}.  
  We will first discuss transcritical 
 bifurcation   of \eqref{tgeq1} 
 when $\gamma$ is zero
 and then  consider the case when
 $\gamma$ satisfies \eqref{gacon2}.
 We will also study transcritical 
  bifurcation of   random periodic 
 (random almost periodic,  random almost
 automorphic)  solutions  
 of \eqref{tgeq1}. 
 
  When $\gamma$ is absent,   equation \eqref{tgeq1}  reduces   to 
 \be
 \label{tseq1}
 {\frac {dx}{dt}} =\lambda x -\beta (t) x^2
  + \delta x  \circ {\frac {d\omega}{dt}},
 \quad  x(\tau) = x_\tau ,
 \quad t>\tau .
 \ee
 This equation is exactly solvable and   for 
      every  $t, \tau  \in \R$
   with $t \ge \tau$, $\omega \in \Omega$
   and $x_\tau \in \R$, the solution is given by
   \be
   \label{tseq4}
   x(t, \tau, \omega, x_\tau)
   = {\frac {x_\tau} {
     e^{\lambda (\tau -t)
   +  \delta (\omega (\tau) - \omega (t) )} 
   +  x_\tau \int_\tau^t e^{ \lambda (r-t) + \delta (\omega (r)
   -\omega (t) ) } \beta (r) dr
   }}.
   \ee 
    It follows   from
     \eqref{tseq4} that  if  $x_0 >0$,  then the 
   solution $ x(t, \tau, \omega, x_0) $  is defined
   for   all   $t \ge \tau$. Similarly, if 
      $x_0 <0$,  then the 
   solution $ x(t, \tau, \omega, x_0) $  is defined
   for   all   $t  \le  \tau$.  
   Based on this fact, we will be
   able  to study   the dynamics of
   \eqref{tseq1} for positive 
   initial  data as $t \to \infty$ as well as
   the dynamics  for negative initial   data
   as
   $t \to -\infty$.
   In the pullback sense,  this
   allows  us   to  explore   the dynamics of solutions
    with positive 
   initial  data as $\tau \to - \infty$  or
   with   negative initial   data
   as
   $ \tau  \to  \infty$.
    By   \eqref{tseq4} we get  that, 
   for each $t \in \R^+$, $\tau \in \R$,
  $ \omega \in \Omega$ 
  and $x_{0} \in \R$, 
 $$
   x(\tau, \tau -t,  \theta_{-\tau}\omega,  x_{0} )
   = {\frac {x_{0}} {
     e^{ - \lambda  t 
   +  \delta   \omega (-t) } 
   +  x_{0} \int_{\tau -t}^\tau 
    e^{ \lambda (r-\tau) + \delta  \omega (r -\tau)
     } \beta (r) dr
   }}
  $$
     \be
   \label{tseq5} 
   = {\frac {x_{0}} {
     e^{ - \lambda  t 
   +  \delta   \omega (-t) } 
   +  x_{0} \int_{ -t}^0
    e^{ \lambda r  + \delta  \omega (r )
     } \beta (r +\tau ) dr
   }}.
   \ee
   
   By \eqref{asyom}  and
   \eqref{tseq5} we  obtain,
   for every $\lambda >0$  and  $x_{0}>0$,
       \be
   \label{tseq6}
      \lim_{t \to \infty}
   x(\tau, \tau -t,  \theta_{-\tau}\omega,  x_{0} )
   =
    {\frac {1} {
        \int_{-\infty}^0
    e^{  \lambda r +  \delta  \omega (r)
     } \beta (r +\tau) dr
     }}.
   \ee
   Analogously,  by \eqref{tseq5}  we  obtain,
     for    every
   $\lambda < 0$  and  $x_{0}<0$,
    \be
   \label{tseq7}
      \lim_{t \to - \infty}
   x(\tau, \tau -t,  \theta_{-\tau}\omega,  x_{0} )
   = 
    {\frac {-1} {
       \int^{\infty}_0
    e^{ \lambda r + \delta  \omega (r)
     } \beta (r +\tau) dr
     }}.
     \ee
     Given  $\lambda \in  \R $, $\tau \in \R$  and
     $\omega \in \Omega$,  we set
     \be
     \label{txplus}
     x_{\lambda} (\tau, \omega) 
      =
   \left \{
   \begin{array}{ll}
      \left (
        \int_{-\infty}^0
    e^{  \lambda r +  \delta  \omega (r)
     } \beta (r +\tau) dr \right )^{-1}
        &  \text{ if } \lambda >0; \vspace{3mm}\\
    - \left (
       \int^{\infty}_0
    e^{ \lambda r + \delta  \omega (r)
     } \beta (r +\tau) dr
     \right )^{-1} &  \text{ if } \lambda <0.
   \end{array}
   \right.
     \ee
     By \eqref{txplus}  we see that
     for every fixed $\tau \in \R$,
     $x_\lambda (\tau, \cdot)$ is measurable.
     By an argument similar    to Lemma \ref{scso1} one
     can verify   that
     $x_\lambda$ is a tempered complete quasi-solution
     of \eqref{tseq1}.
     
     \begin{thm}
     \label{thmst} 
     If     \eqref{bcon1}  holds, 
    then   the random complete quasi-solutions of \eqref{tseq1}
     undergo a stochastic transcritical bifurcation
     at $\lambda =0$. More precisely:
     
     (i)  If $\lambda <0$,   then   \eqref{tseq1} has 
     two random complete quasi-solutions $x=0$   and 
     $x=x_\lambda$ given by \eqref{txplus}.  
      The zero solution is asymptotically stable in $(0, \infty)$
      and    pullback attracts  every  compact subset $K$ of $(0,\infty)$,
      i.e., 
     \be\label{thmst1}
         \lim_{t \to \infty}
   x(\tau, \tau -t,  \theta_{-\tau}\omega,  K )
   =0.
   \ee
     Moreover,  for all $\tau \in \R$  and $\omega
   \in \Omega$,  we have
   $x_\lambda (\tau, \omega) <0$ and 
   \be\label{thmst2}
   \lim_{\lambda \to 0} x_\lambda (\tau, \omega) = 0.
   \ee
    
     (ii)  If $\lambda >0$,  then   \eqref{tseq1} has 
     two random complete quasi-solutions $x=0$   and 
     $x=x_\lambda$ given \eqref{txplus}.
      The zero solution is  unstable   in $(0, \infty)$
      and  $x_\lambda$  pullback attracts  every  compact subset $K$ of $(0,\infty)$,
      i.e., 
     \be\label{thmst3}
         \lim_{t \to \infty}
   x(\tau, \tau -t,  \theta_{-\tau}\omega,  K )
   = x_\lambda(\tau, \omega).
   \ee
     Moreover,  for all $\tau \in \R$  and $\omega
   \in \Omega$, we have  $x_\lambda (\tau, \omega) >0$  and
    \be\label{thmst4}
   \lim_{\lambda \to 0} x_\lambda (\tau, \omega) = 0.
  \ee
      \end{thm}

     \begin{proof}
     (i)  Let $K$  be a compact subset of $(0, \infty)$. Then by \eqref{tseq5}
     we get
     \be\label{pthmst_1}
     \limsup_{t \to \infty}
     \sup_{x_0 \in K }
      x(\tau, \tau -t,  \theta_{-\tau}\omega,  x_{0} )
   \le  \limsup_{t \to \infty} {\frac {1} { 
     \int_{ -t}^0 
    e^{ \lambda r + \delta  \omega (r )
     } \beta (r +\tau) dr
   }} .
   \ee
   On the other hand, 
   by  \eqref{bcon1}  and  \eqref{asyom} 
   we get,  for all $\tau \in \R$  and $\omega \in \Omega$,
   \be\label{pthmst_2}
    \lim_{t \to \infty} {\frac {1} { 
     \int_{ -t}^0 
    e^{ \lambda r + \delta  \omega (r )
     } \beta (r +\tau) dr
   }} =0.
   \ee
   By  \eqref{pthmst_1}-\eqref{pthmst_2} we obtain
   \eqref{thmst1}.
   The asymptotic stability of $x=0$ in $(0, \infty)$
   and the convergence \eqref{thmst2} can be proved
   by a   argument  similar to Theorem \ref{thmgs}. 
   
   (ii) Let $K=[a,b]$   with $a>0$. By 
   \eqref{tseq5} we have,  for all
   $x_0 \in K$,
     \be\label{pthmst_3}
    x(\tau, \tau -t,  \theta_{-\tau}\omega,  x_{0} )
    \ge
    {\frac {1} {
     e^{ - \lambda  t 
   +  \delta   \omega (-t) } a^{-1} 
   +   \int_{ -t}^0
    e^{ \lambda r  + \delta  \omega (r )
     } \beta (r +\tau ) dr,
   }}  
   \ee
   and 
    \be\label{pthmst_4}
    x(\tau, \tau -t,  \theta_{-\tau}\omega,  x_{0} )
   \le 
    {\frac {1} {
     e^{ - \lambda  t 
   +  \delta   \omega (-t) } b^{-1} 
   +   \int_{ -t}^0
    e^{ \lambda r  + \delta  \omega (r )
     } \beta (r +\tau ) dr
   }}.
\ee
 Since $\lambda>0$, by 
   \eqref{bcon1}  and  \eqref{asyom} we find
   that for all $\tau \in \R$ and
   $\omega \in \omega$, 
   the right-hand sides of \eqref{pthmst_3}
   and \eqref{pthmst_4} converge
   to  $x_\lambda(\tau, \omega)$ as $t \to \infty$,
   which implies \eqref{thmst3}.
   Note that    the instability of $x=0$
   in $(0, \infty)$ 
   is implied by 
   \eqref{thmst3}. We then  conclude   the proof.
      \end{proof}

      By \eqref{txplus} we see that if
      $\beta$ is a periodic function with period $T>0$,
      then so is $x_\lambda (\cdot, \omega)$
      for all $\omega \in \Omega$. 
      By an argument similar to Lemmas \ref{apb1}
      and \ref{aab1}, one can prove
      $x_\lambda (\cdot, \omega)$
      is   almost periodic (almost automorphic)
      provided $\beta$ is 
         almost periodic (almost automorphic).
         Based on this fact,  we have
         the following results   from
         Theorem \ref{thmst}.
         
          \begin{cor}
     \label{2thmst} 
    Suppose     \eqref{bcon1}  holds
     and $\beta: \R \to \R$
     is periodic   (almost periodic, almost automorphic).
   Then   the random  periodic 
   (almost periodic,  almost automorphic) solutions
     of \eqref{tseq1}
     undergo a stochastic transcritical bifurcation
     at $\lambda =0$. More precisely:
     
     (i)  If $\lambda <0$,   then   \eqref{tseq1} has 
     two   random periodic 
   (almost periodic,  almost automorphic) solutions
       $x=0$   and 
     $x=x_\lambda$ given by \eqref{txplus}.  
      The zero solution is asymptotically stable in $(0, \infty)$
      and  \eqref{thmst1}-\eqref{thmst2} are fulfilled. 
       
        (i)  If $\lambda > 0$,   then   \eqref{tseq1} has 
     two   random periodic 
   (almost periodic,  almost automorphic) solutions
       $x=0$   and 
     $x=x_\lambda$ given by \eqref{txplus}.  
      The zero solution is  unstable in $(0, \infty)$
      and  \eqref{thmst3}-\eqref{thmst4} are fulfilled. 
    \end{cor}

    Next, we consider bifurcation of \eqref{tgeq1}
    with $\gamma$ satisfying \eqref{gacon2}.
    In this case, we can associate two 
    exactly solvable 
    systems with \eqref{tgeq1}.
    Given $t, \tau \in \R$
    with $t> \tau$,  consider
    \be
 \label{ltgeq1}
 {\frac {dx}{dt}} =\lambda x - (\beta (t) -c_2)  x^2
  + \delta x  \circ {\frac {d\omega}{dt}},
 \quad  x(\tau) = x_\tau , \quad t>\tau,
 \ee
 and
   \be
 \label{ltgeq2}
 {\frac {dx}{dt}} =\lambda x - (\beta (t) -c_1)  x^2
  + \delta x  \circ {\frac {d\omega}{dt}},
 \quad  x(\tau) = x_\tau , \quad t>\tau.
 \ee
 By \eqref{gacon2} we find   that
 the solutions of \eqref{ltgeq1}
 and \eqref{ltgeq2} are
 super- and sub-solutions of
 \eqref{tgeq1}, respectively.
 The random complete quasi-solutions of
 \eqref{ltgeq1} and \eqref{ltgeq2}
 can be studied  as equation \eqref{tseq1}.
 Then by the comparison principle
 and the arguments discussed in the 
 previous  section, 
 we can obtain transcritical bifurcation
for  \eqref{tgeq1}. We here just 
present the results and
 will not repeat the details in this case.

     \begin{thm}
     \label{3thmst} 
     If     \eqref{bcon1} and \eqref{gacon2}  hold, 
    then   the random complete quasi-solutions of \eqref{tgeq1}
     undergo a stochastic transcritical bifurcation
     at $\lambda =0$. More precisely:
     
     (i)  If $\lambda <0$,   then   \eqref{tgeq1} has 
     two random complete quasi-solutions $x=0$   and 
     $x=x_\lambda$
     with $x_\lambda (\tau, \omega)<0$
     for all $\tau\in\R$  and $\omega \in \Omega$. 
      The zero solution is asymptotically stable in $(0, \infty)$
       and  \eqref{thmst1}-\eqref{thmst2} are fulfilled.

     (ii)  If $\lambda >0$,  then   \eqref{tgeq1} has 
     two random complete quasi-solutions $x=0$   and 
     $x=x_\lambda$
     with $x_\lambda (\tau, \omega)>0$
     for all $\tau\in\R$  and $\omega \in \Omega$. 
      The zero solution is  unstable   in $(0, \infty)$
      and  \eqref{thmst3}-\eqref{thmst4} are fulfilled.

      If, in addition, $\beta$  is a  periodic
function,  then so is
      $x_\lambda$  for $\lambda \neq 0$. 
      \end{thm}

\end{document}